\documentclass{amsart}
\usepackage{amsaddr}

\usepackage[margin=3.6cm]{geometry}
\usepackage[dvipdfmx]{graphicx}
\usepackage{tikz,enumerate}

\newtheorem{theorem}{Theorem}[section]
\newtheorem{lemma}[theorem]{Lemma}

\theoremstyle{definition}
\newtheorem{definition}[theorem]{Definition}

\theoremstyle{remark}

\numberwithin{equation}{section}

\newcommand{\R}{\mathbb{R}}

\newcommand{\eps}{\varepsilon}
\newcommand{\cD}{\mathcal{D}}

\newcommand{\inn}{\textrm{in}}
\newcommand{\out}{\textrm{out}}
\newcommand{\loc}{\textrm{loc}}

\DeclareMathOperator{\sign}{sign}
\DeclareMathOperator{\scal}{scal}
\DeclareMathOperator{\Vol}{Vol}
\DeclareMathOperator{\supp}{supp}

\begin{document}

\title[Infinite-time incompleteness of noncompact Yamabe flow]{Infinite-time incompleteness of noncompact\\ Yamabe flow}

\author[J. Takahashi]{Jin Takahashi}
\address{Department of Mathematical and Computing Science, \\
Tokyo Institute of Technology, Tokyo 152-8552, Japan}
\email{takahashi@c.titech.ac.jp}
\thanks{The first author was supported in part 
by JSPS KAKENHI 
Grant-in-Aid for Early-Career Scientists 19K14567. 
The second author was supported in part by JSPS KAKENHI 
Grant-in-Aid for Early-Career Scientists 18K13415.}

\author[H. Yamamoto]{Hikaru Yamamoto}
\address{Department of Mathematics, Faculty of Pure and Applied Science, \\
University of Tsukuba, Ibaraki 305-8571, Japan}
\email{hyamamoto@math.tsukuba.ac.jp}


\begin{abstract}
We show the noninheritance of 
the completeness of the noncompact Yamabe flow. 
Our main theorem states the existence of a long time solution 
which is complete for each time 
and converges to an incomplete Riemannian metric. 
This shows the occurrence of the infinite-time incompleteness. 
\end{abstract}

\maketitle

\section{Introduction}\label{intro}

We study the inheritance of the completeness of the Yamabe flow 
\begin{equation}\label{y-flow}
	\partial_t  g_{t}= -\scal(g_{t})g_{t}, 
\end{equation}
where $\scal(g_{t})$ denotes the scalar curvature of $g_{t}$. 
In the main theorem, we construct a long time solution $(M,g_{t})$ 
which is a complete Riemannian manifold for each $t\in[0,\infty)$, 
and $g_t$ converges to an incomplete Riemannian metric on $M$ as $t\to\infty$. 
Hence, the completeness is not inherited to the limit. 
The proof relies on the study of singular solutions 
of the fast diffusion equation. 
We note that Section \ref{sec:exfas} is devoted 
to the analysis of the equation 
and can be read separately. 
In what follows, 
we briefly review known results, and then we state the main result.

\subsection{Compact Yamabe flow}
The Yamabe flow is an evolving Riemannian metric $g_{t}$ 
on a manifold $M$ satisfying \eqref{y-flow}. 
It was introduced by Hamilton \cite{H89} for a resolution of the Yamabe problem. 
The Yamabe problem, for a given Riemannian metric $g_{0}$, is asking the existence of a minimizer of the Yamabe functional 
\[E(g):=\frac{\int_{M}\scal(g)dv_{g}}{{\Vol(M,g)}^{\frac{n-2}{n}}}\]
among the conformal class of $g_{0}$, 
provided that $M$ is compact and $n$-dimensional with $n\geq 3$, 
where $dv_{g}$ is the volume form of $g$ and $\Vol(M,g)$ is the volume of $(M,g)$. 
The negative gradient flow of $E$ is 
\[
	\partial_t g_{t}
	= \left(-\scal(g_{t})+\int_{M}\scal(g_{t}) dv_{g_{t}}\right)g_{t}
\]
by scaling the time variable, 
and 
this is called the normalized  Yamabe flow, compared with \eqref{y-flow}. 
Even the normalized Yamabe flow is better to obtain a long time solution (see \cite{Y94} for instance), we prefer to study the unnormalized flow \eqref{y-flow} 
since $M$ is noncompact in this paper. 

The flow \eqref{y-flow} is called a compact (resp. noncompact) Yamabe flow 
when the underlying manifold $M$ is compact (resp. noncompact). 
The study of compact Yamabe flows is already developed deeply. 
The existence of a long time solution and 
the convergence to a constant scalar curvature metric were settled by 
Hamilton \cite{H89}, Chow \cite{C92}, Ye \cite{Y94}, 
Schwetlick and Struwe \cite{SS03} and Brendle \cite{B05} 
with making assumptions on the curvature of the initial metric weaker and weaker. 
We remark that the flows treated in these papers are the normalized one. 
The unnormalized flow \eqref{y-flow} 
may develop singularities in finite time even in the compact case. 
The study related to asymptotic behavior at the singularities, 
including the classification of singularities, 
ancient solutions and solitons, has been recently proceeded 
by Daskalopoulos, del Pino and Sesum \cite{DDS18} and 
Daskalopoulos, del Pino, King and Sesum \cite{DDKS16,DDKS17}. 
Indeed, there are many other papers contributing the compact case. 
We refer the reader to the above papers and the references given there.

\subsection{Noncompact Yamabe flow}
The noncompact case is complicated. 
There are many unexpected phenomena 
from the viewpoint of the compact Yamabe flow. 
Even the number of papers on the noncompact Yamabe flow is still limited, 
the knowledge has been gradually accumulated. 
When we study a noncompact Riemannian manifold, 
the completeness of the initial Riemannian manifold $(M,g_{0})$ is important. 
We review some works which hold without 
any assumptions on the topological type of $M$. 
The first contribution was made by Ma and An \cite{MA99}. 
They proved the short time existence of \eqref{y-flow} 
in the case where the initial Riemannian manifold is 
complete and locally conformally flat with Ricci curvature bounded from below, 
and also proved the existence of a long time solution in the case where 
the initial Riemannian manifold is complete 
with the Ricci curvature and the nonpositive scalar curvature 
bounded from below. 
Some criterions for the extension of the flow are also known. 
Ma, Cheng and Zhu \cite[Theorem 3]{MCZ12} proved that 
if $(M,g_{0})$ is a complete Riemannian manifold 
with bounded curvature and with a 
positive Yamabe constant, 
then a Yamabe flow $g_{t}$ ($t\in [0,T)$ with $T<\infty$) 
starting from $g_0$ satisfying 
$\|\scal(g_{t})\|_{L^{(n+2)/2}(M\times(0,T))}<\infty$ 
can be extended beyond the time $T$. 
Recently, Ma \cite{M19} also proved the long time existence 
and convergence of the Yamabe flow 
starting from an initial complete Riemannian metric with nonnegative scalar curvature 
under the assumption on the existence of a kind of lower solutions. 

In the above results, the initial Riemannian manifolds are complete. 
In contrast, the Yamabe flow starting from an incomplete Riemannian manifold is also interesting. 
A series of works by Giesen and Topping \cite{T10, GT10, GT11, T15} revealed the so-called \emph{instantaneously completeness}. 
They studied two-dimensional Ricci flows, note that Ricci flow and Yamabe flow coincide in the two-dimensional case, 
and they proved the existence and uniqueness of a flow 
starting from an incomplete initial Riemannian manifold so that 
it turns to be complete for positive time. 
Recently, Schulz generalized these results to the higher dimensional Yamabe flow in several settings \cite{S19, S20i, S20u, S20}. 
For instance, in \cite{S20u}, for any conformal initial metric $g$ 
on the hyperbolic space $\mathbb{H}$ with $\dim\mathbb{H}\geq 3$, 
it was proved 
that there exists an instantaneously complete Yamabe flow starting from $g$. 
These results, at least, indicate that 
completeness and incompleteness are subtle conditions along the Yamabe flow.

\subsection{Main result}
When we restrict the topological type of $(M,g_{0})$, we can say more. 
A typical case is a manifold with ends. 
For instance, the Euclidean space is a manifold with one end and a cylinder is a manifold with two ends. 
In this direction, for instance, Chen and Zhu \cite{CZ15}, Ma \cite{M21} and Chen and Wang \cite{CW21} 
studied Riemannian manifolds with one end which is asymptotically flat or Euclidean. 
These results are strongly related to ADM mass. 
Also in the case with one end, 
Daskalopoulos and Sesum \cite{DS08} 
proved the existence of a Riemannian metric $g_{0}$ on $\mathbb{R}^{n}$ 
satisfying the following properties: 
(1) complete, (2) asymptotically cylindrical, 
(3) the solution $g_{t}$ of \eqref{y-flow} starting from $g_{0}$ 
develops the singularity at some finite time $T^{\ast}$ 
and (4) $g_{t}$ becomes incomplete at some time $T\in (0,T^{\ast})$. 
Even though $g_t$ has more properties, but we focus on the above ones. 
We would like to call this phenomenon the \emph{finite-time incompleteness}. 
In this paper, in contrast to the finite-time incompleteness, 
we give an example of a Riemannian manifold $(M,g_{0})$ 
which causes the ``infinite''-time incompleteness 
considering a manifold with two ends having different 
asymptotical types. 
Here we say that a Yamabe flow $(M,g_{t})$ causes the \emph{infinite-time incompleteness} if it is defined for all $t\in[0,\infty)$, 
$g_{t}$ is complete for all $t\in[0,\infty)$ 
and $g_{t}$ converges to an incomplete Riemannian metric 
$g_{\infty}$ on $M$ as $t\to \infty$.

\begin{theorem}\label{mainthm}
Let $n\geq 3$ and $M:=\mathbb{R}^{n}\setminus\{0\}$ $(\cong S^{n-1}\times\mathbb{R})$. 
Let $E_{1}$ and $E_{2}$ be the ends of $M$ around $|x|=\infty$ and $|x|=0$, respectively. 
For each $c_{1}, c_{2}> 0$ and $(n+2)/2<\lambda<n+2$, 
there exists a Riemannian metric $g_{0}=g_{0}(c_{1},c_{2},\lambda)$ on $M$ 
satisfying the following conditions from {\rm (1)} to {\rm (4)}: 
\begin{enumerate}[(1)]
\item The end $E_{1}$ is asymptotically Euclidean of order $(n-2)\lambda/(n+2)-2>0$ if $n\geq 6$. 
\item The end $E_{2}$ is asymptotically conical of order $\tau$ with rink $(S^{n-1},Bg_{S^{n-1}})$, 
where the constants $B\in(0,1)$ and $\tau>0$ 
are explicitly determined by $n$ and $\lambda$ 
in \eqref{Btau}. 
\item The scalar curvature of $g_{0}$ is negative. 
\item 
There exists a unique long time solution $g_{t}$ $(t\in[0,\infty))$ of the Yamabe flow \eqref{y-flow} starting from $g_{0}$, 
and $g_{t}$ satisfies the following properties: 
\begin{enumerate}
\item $(M,g_{t})$ is a complete Riemannian manifold
and is conformally equivalent 
to the standard metric $g_{\mathbb{R}^{n}}$
on $\mathbb{R}^{n}\setminus\{0\}$ for all $t\in[0,\infty)$. 
\item 
The volume form $\mathop{\mathrm{vol}}(g_{t})$ of $g_t$ satisfies 
$|x|^{2n\lambda/(n+2)}\mathop{\mathrm{vol}}(g_{t})\to c_{1}^{2n/(n+2)}\mathop{\mathrm{vol}}(g_{\mathbb{R}^{n}})$ as $|x| \to 0$ and 
$\mathop{\mathrm{vol}}(g_{t})\to c_{2}^{2n/(n+2)}\mathop{\mathrm{vol}}(g_{\mathbb{R}^{n}})$ as $|x| \to \infty$
for all $t\in [0,\infty)$. 
\item 
The Yamabe constant $Y(M,g_{t})$ of $(M,g_{t})$ 
is equal to $Y(S^{n},g_{S^{n}})>0$ for all $t\in[0,\infty)$. 
\item 
$g_{t}$ converges to $g_{\infty}:=c_{2}^{4/(n+2)}g_{\mathbb{R}^{n}}$ 
in $C^2_\loc(\R^n\setminus\{0\})$ 
as $t\to\infty$. 
Especially, $(M,g_{\infty})$ is incomplete. 
\end{enumerate}
\end{enumerate}
\end{theorem}

We remark that the conditions from (1) to (3) are only for the initial metric and (4) is for the flow itself and its limit. 
The shape of $(M,g_{0})$ is similar to a ``coffee dripper'' (see Figure \ref{fig1}). 
We remark that some terminologies in the above theorem are defined 
in Section \ref{FATG}: 
a manifold with ends, asymptotically Euclidean or conical ends 
and the Yamabe constant are 
defined in Definitions \ref{mwithe}, \ref{AEAC} and \ref{Yaminv}, 
respectively. 
We also remark that from (1) and (2), 
it follows that $(M,g_{0})$ is asymptotically flat for $n\geq 6$. 

As far as the authors know, this is the first example of the long time solution 
of the noncompact Yamabe flow 
which causes the infinite time incompleteness. 
This implies that the completeness in positive time is not inherited to the limit. 
Also, to the best of the authors' knowledge, this is the first paper constructing a nontrivial solution of the Yamabe flow starting from a complete Riemannian manifold having more than one end.

\begin{figure}
\begin{center}
\includegraphics[bb=146 558 449 691, clip , scale=1.08]{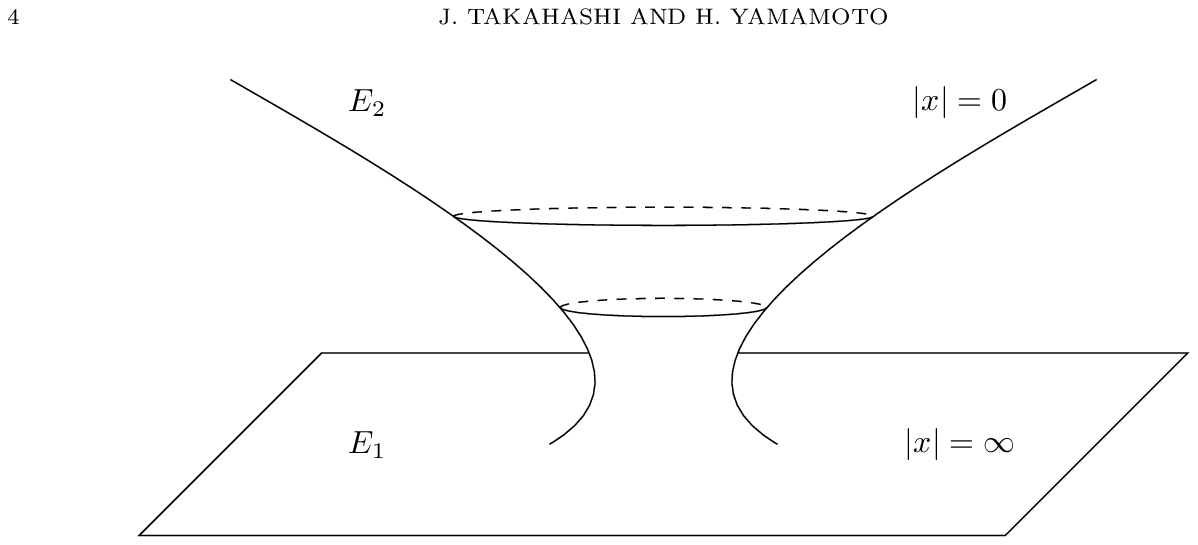}
\end{center}
\caption{Abstract shape of $(M,g_{0})$ in Theorem \ref{mainthm}}
\label{fig1}
\end{figure}

We remark that our initial Riemannian manifold $(M,g_{0})$ 
is applicable to $(M,g_{0})$ in Ma and An \cite[Theorem 1]{MA99} 
and also applicable to $(M,g_{M})$ in Schulz \cite[Theorem 1]{S20}. 
Hence, the existence of a long time solution starting 
from our $(M,g_{0})$ follows from their general theorems. 
One of the novelties of our result is a discovery 
of a complete initial metric such that its long time limit 
along the Yamabe flow exists and is incomplete.

Even the existence part follows from previous works, 
we give a purely analytical proof of the existence 
to obtain precise asymptotic behaviors of the solution for our use. 
The argument is based 
on the work of V\'azquez and Winkler \cite{VW11} 
for the fast diffusion equation 
and on the construction of suitable comparison functions 
in the same spirit of 
Fila, the first author and Yanagida \cite{FTY19}
and the same authors with Mackov\'a \cite{FMTY20}. 
Hence, readers who want to know our results 
on the fast diffusion equation from the viewpoint of PDE 
can read Section \ref{sec:exfas} directly and separately. 
As results purely in PDE,
the main results in Section \ref{sec:exfas} have some novelties
and seem to be of independent interest.

\subsection{Reduction to the fast diffusion equation}\label{Yamabe-fast}
This subsection is a bypass from geometry to analysis. 
In particular, we reduce the Yamabe flow \eqref{y-flow} on $\mathbb{R}^{n}\setminus\{0\}$ to the fast diffusion equation, 
and in Section \ref{sec:exfas} we prove the uniqueness and existence of a long time solution of the fast diffusion equation 
with a specific singular initial data 
and identify the exact behaviors 
near the singularity of the solution and of its gradient.

Let $(M,g_{0})$ be an $n$-dimensional Riemannian manifold and $v$ be a smooth positive function on $M$. Throughout this paper, we assume $n\geq 3$ 
and put 
\[
	m_{c}:=\frac{n-2}{n+2}\in(0,1).
\]
It is well-known that the scalar curvature of $g:=v^{4/(n-2)}g_{0}$ is given by 
\begin{equation}\label{scalc}
\scal(g)=v^{-\frac{n+2}{n-2}}\left(-\frac{4(n-1)}{n-2}\Delta_{g_{0}}v+\scal(g_{0})v\right), 
\end{equation}
where $\Delta_{g_{0}}$ is the Laplacian with respect to $g_{0}$. 
Since the Yamabe flow $g_{s}$ preserves the conformal class, 
we can assume that $g_{s}$ is written as $v_{s}^{4/(n-2)}g_{0}$ 
with some time dependent smooth positive function $v_{s}$ on $M$. 
Then, the equation \eqref{y-flow} is reduced to the PDE for $v_{s}$ as 
\begin{equation}\label{y-flow2}
\partial_{s}\left(v_{s}^{\frac{4}{n-2}}\right)=-v_{s}^{-1}\left(-\frac{4(n-1)}{n-2}\Delta_{g_{0}}v_{s}+\scal(g_{0})v_{s}\right). 
\end{equation}
By setting 
\begin{equation}\label{uv}
u(\,\cdot\,,t):=v^{\frac{n+2}{n-2}}(\,\cdot\,,s(t))\qquad\mbox{ with }s(t):=\frac{n-2}{(n-1)(n+2)}t, 
\end{equation}
we see that the equation \eqref{y-flow2} is transformed into 
\begin{equation}\label{y-flow3}
\partial_t u=\Delta_{g_{0}}(u^{\frac{n-2}{n+2}})-\frac{n-2}{4(n-1)}\scal(g_{0})u^{\frac{n-2}{n+2}}. 
\end{equation}
Especially, if $(M,g_{0})$ is scalar flat, 
the equation \eqref{y-flow3} coincides with the fast diffusion equation: 
\[
	\partial_t u=\Delta_{g_{0}}(u^{m_{c}}). 
\]
In this paper, we apply this reduction to the Yamabe flow on $(\mathbb{R}^{n}\setminus\{0\},g_{\mathbb{R}^{n}})$, 
where $g_{\mathbb{R}^{n}}$ is the standard metric on $\mathbb{R}^{n}$, 
of course, it is scalar flat. 
Hence, the bulk of this paper is the analysis of the fast diffusion equation which is done in Section \ref{sec:exfas}. 
After doing that, in Section \ref{FATG}, we recover the geometric meaning of the corresponding metric from the solution, that is, geometry from analysis. 

\subsection{The initial metric}
We can write the initial metric $g_{0}$ 
on $M=\mathbb{R}^{n}\setminus\{0\}$ explicitly. 
For fixed constants $c_{1}, c_{2}>0$ and $(n+2)/2<\lambda<n+2$, 
define $u_{0}\in C^{\infty}(\mathbb{R}^{n}\setminus\{0\})$ by 
\[
	u_{0}(x):= ( c_1^{m_{c}} |x|^{-{m_{c}}\lambda} 
	+ c_2^{m_{c}} )^\frac{1}{m_{c}}. 
\]
Then, indeed, the initial metric $g_{0}$ in Theorem \ref{mainthm} is given by 
\begin{equation}\label{initialg0}
g_{0}:=u_{0}^{\frac{4}{n+2}}g_{\mathbb{R}^{n}}. 
\end{equation}
We will prove that $(\mathbb{R}^{n}\setminus\{0\},g_{0})$ satisfies the conditions from (1) to (4) in Section \ref{FATG}. 
Hence, solving the Yamabe flow equation \eqref{y-flow} 
with initial metric $g_{0}$ is equivalent to solving 
the following initial value problem of the fast diffusion equation 
\[
\left\{
\begin{aligned}
	&\partial_t u=\Delta (u^{m_c}), 
	&&x\in \R^n \setminus \{0\},\;  t>0, \\
	& u(x,0) = u_0(x), 
	&&x\in \R^n \setminus \{0\},  
\end{aligned} 
\right .
\]
where $\Delta$ is the standard Laplacian on $\mathbb{R}^{n}$. 
The existence and uniqueness of a solution of this PDE 
and its properties are studied in the next section.

\subsection{Organization of this paper}
Section \ref{intro} is the introduction to the Yamabe flow. 
We start with the brief review of the compact case. 
We refer some results on the noncompact case 
and state our main theorem. 
The introduction is ended with the reduction from the Yamabe flow to the fast diffusion equation. 
Section \ref{sec:exfas} is a purely analysis section. 
In this section, 
we analyze the fast diffusion equation 
with a specific singular initial function. 
Section \ref{FATG} is completely geometry section. 
We interpret the properties of the solution 
into the sense of Riemannian geometry. 
Then, the properties in Theorem \ref{mainthm} are confirmed step by step.

\section{Singular solutions in the fast diffusion equation}\label{sec:exfas}
In this section, 
we analyze the following problem of the fast diffusion equation 
\begin{equation}\label{eq:fast}
\left\{
\begin{aligned}
	&\partial_t u=\Delta (u^m), 
	&&x\in \R^n \setminus \{0\},\; t>0, \\
	& u(x,0) = u_0(x) := 
	( c_1^m |x|^{-m\lambda} + c_2^m )^\frac{1}{m}, 
	&&x\in \R^n \setminus \{0\}, 
\end{aligned} 
\right.
\end{equation}
where $n\geq3$, $c_1>0$ and $c_2\geq0$. 
Throughout this section, we assume that 
\begin{equation}\label{eq:lamcon}
	0<m<\frac{n-2}{n}, \qquad 
	\frac{2}{1-m} < \lambda < \frac{n-2}{m}. 
\end{equation}
In particular, the results in this section are also valid 
for a geometrically important case 
\[
	m=m_c=\frac{n-2}{n+2}, \qquad 
	\frac{n+2}{2} < \lambda < n+2, 
	\qquad c_1, c_2>0. 
\]
For $0<T\leq \infty$, we write 
\[
	X_T:=C^{2,1}((\R^n\setminus\{0\})\times(0,T))
	\cap C((\R^n\setminus\{0\})\times[0,T)). 
\]
By a solution of \eqref{eq:fast}, 
we mean a nonnegative function 
$u\in X_T$ 
which satisfies \eqref{eq:fast} for some $0<T\leq \infty$. 
If a solution $u$ satisfies 
$u(x,t)\to\infty$ for each $t\in[0,T)$ as $x\to0$, 
we call $u$ a singular solution. 
The main result in this section shows 
the existence and uniqueness of a singular solution and 
its precise properties.

\begin{theorem}\label{th:blowdown}
Let $n\geq3$, $c_1>0$ and $c_2\geq0$. 
Assume that $m$ and $\lambda$ satisfy \eqref{eq:lamcon}. 
Then, there exists a unique solution $u$ of the problem \eqref{eq:fast}. 
Moreover, the solution $u$ is a positive  long time solution and 
belongs to $C^{2,1}((\R^n\setminus\{0\})\times(0,\infty))$ and 
$C((\R^n\setminus\{0\})\times[0,\infty))$. 
Furthermore, $u$ satisfies the following 
properties from {\rm(i)} to {\rm(vi)}: 
\begin{enumerate}
\item[(i)]
$u(x,t)$ is radially symmetric and nonincreasing in $|x|$
for each $t\in[0,\infty)$. 
\item[(ii)]
$u(x,t)$ is nonincreasing in $t\in[0,\infty)$ 
for each $x\in\R^n\setminus\{0\}$. 
\item[(iii)]
$|x|^\lambda u(x,t)\to c_1$ locally uniformly 
for $t\in[0,\infty)$ as $|x|\to0$. 
\item[(iv)]
$|x|^{\lambda+1} (x/|x|)\cdot \nabla u(x,t)\to -c_1\lambda$ 
locally uniformly for $t\in[0,\infty)$ as $|x|\to0$. 
\item[(v)]
$u(x,t)\to c_2$ uniformly for $t\in[0,\infty)$ as $|x|\to\infty$. 
\item[(vi)]
$u(\,\cdot\,,t) \to c_2$ in $C^2_\loc(\R^n\setminus\{0\})$ as $t\to\infty$. 
\end{enumerate}
\end{theorem}

Before starting the proof, we introduce the work of 
V\'azquez and Winkler \cite{VW11} 
and state the novelty of Theorem \ref{th:blowdown}. 
In \cite[Section 3]{VW11}, 
they studied the radial proper solution $U$ of 
\begin{equation}\label{eq:fastbdd}
\left\{ 
\begin{aligned}
	&\partial_t U= m^{-1} \Delta U^m, &&|x|<R, \;t>0, \\
	&U(x,t)= c_1 R^{-\lambda}, &&|x|=R, \;t>0, \\
	&U(x,0)= c_1 |x|^{-\lambda}, &&|x|<R,  \\
\end{aligned}
\right. 
\end{equation}
where $0<m<(n-2)/n$, $R>0$, $c_1>0$ and $\lambda>0$, 
and the proper solution is defined 
by the limit of solutions of appropriate regularized problems. 
In the case where $\lambda$ satisfies \eqref{eq:lamcon}, 
they proved in \cite[Lemma 3.9]{VW11} that 
$\partial_t U\leq0$ 
and $U(\,\cdot\,,t) \to c_1 R^{-\lambda}$ 
in $C^2_\loc(\R^n\setminus\{0\})$ as $t\to\infty$. 
In particular, they found the occurrence of 
the infinite-time blow-down of the solution. 
By \cite[Lemmas 3.1, 3.8]{VW11} and their proofs, 
for $0<T<\infty$, there exists $c_T>0$ depending on $T$ such that 
\[
	c_T |x|^{-\lambda} \leq U(x,t) \leq c_1|x|^{-\lambda}, 
	\qquad 0<|x|\leq R, \; 0\leq t\leq T. 
\]
They also analyzed the evolution of singularities 
in the other ranges 
$0<\lambda\leq 2/(1-m)$ and $\lambda\geq (n-2)/m$ 
and studied problems on general bounded domains 
with inhomogeneous Dirichlet boundary conditions. 
See also \cite{VW11h} for oscillating phenomena.

The problem \eqref{eq:fast} is a kind of 
the whole space case of \eqref{eq:fastbdd} 
and the proofs in this section can be applied to \eqref{eq:fastbdd}
with simple modifications. 
Thus, our main contribution in Theorem \ref{th:blowdown} 
is the identification of the exact behaviors near the singularity 
of the solution and of its gradient. 
In particular, we construct subsolutions different 
from \cite{VW11} 
in the same spirit of \cite{FMTY20,FTY19} 
and give precise lower bound of solutions, 
see Lemma \ref{lem:loap} for details. 

It seems that 
the uniqueness part in Theorem \ref{th:blowdown} is also new 
compared with the classical result of Herrero and Pierre \cite{HP85} and 
the recent result of Hui \cite{Hu20}. 
In \cite{HP85}, they proved the uniqueness 
of possibly sign-changing solutions 
in the class $u\in C([0,\infty); L^1_\loc(\R^n))$
with $\partial_t u\in L^1_\loc(\R^n\times (0,\infty))$. 
In \cite{Hu20}, the uniqueness of positive singular solutions 
was shown for solutions satisfying 
an estimate of the form 
$C_1 |x|^{-\lambda_1} \leq u(x,t) \leq C_2 |x|^{-\lambda_2}$ 
near the origin for any $t\in(0,T)$, 
see \cite[(1.13)]{Hu20}. 
On the other hand, Theorem \ref{th:blowdown} 
does not require assumptions on 
the behavior of $u(x,t)$ with $t>0$ 
or the integrability around the singular point $x=0$. 
Namely, we prove the following theorem 
based on the argument of Herrero and Pierre \cite[Theorem 2.3]{HP85}.

\begin{theorem}\label{th:uni}
Let $n\geq3$, $0<m<(n-2)/n$ and $0<T\leq \infty$. 
Assume that functions 
$u_1, u_2\in C^{2,1}((\R^n\setminus\{0\})\times(0,T))
\cap C((\R^n\setminus\{0\})\times[0,T))$ satisfy 
\begin{equation}\label{eq:fasteq}
	\partial_t u=\Delta(u |u|^{m-1})
	\qquad \mbox{ in } (\R^n\setminus\{0\})\times(0,T) 
\end{equation}
and 
$u_1(\,\cdot\,,0)=u_2(\,\cdot\,,0)$ on $\R^n \setminus \{0\}$. 
Then, $u_1\equiv u_2$ on $(\R^n\setminus\{0\})\times[0,T)$. 
\end{theorem}

The rest of this section is organized as follows. 
In Subsection \ref{subsec:exsin}, we prove the existence of 
a long time singular solution 
and also prove the properties (i) and (ii) of Theorem \ref{th:blowdown}. 
In Subsection \ref{subsec:comuni}, 
we show Theorem \ref{th:uni}. 
In Subsection \ref{subsec:estisol}, 
we give estimates of the solution and show (iii), (iv) and (v). 
Finally, we study the convergence property (vi) 
in Subsection \ref{subsec:conv} 
and complete the proof of Theorem \ref{th:blowdown}.

\subsection{Existence and monotonicity}\label{subsec:exsin}
To construct a solution of \eqref{eq:fast}, 
we consider the following approximate problem 
\begin{equation}\label{eq:fastap}
\left\{
\begin{aligned}
	&\partial_t u=\Delta u^m, 
	&&x\in \R^n,\; t>0, \\
	& u(x,0) = u_{0,\eps}(x):=
	( c_1^m(|x|^2+\eps)^{-\frac{1}{2}m \lambda} 
	+ c_2^m+ \eps )^\frac{1}{m}, 
	&&x\in \R^n, 
\end{aligned} 
\right.
\end{equation}
with a parameter $0<\eps<1$. 
Since $\eps\leq u_{0,\eps} \leq (c_1^m \eps^{-m\lambda/2}+c_2^m+\eps)^{1/m}$, 
this problem has a unique long time 
solution $u_\eps \in X_\infty$ 
satisfying 
\begin{equation}\label{eq:roughes}
	\eps\leq u_\eps(x,t) \leq (c_1^m \eps^{-\frac{1}{2}m\lambda}+ c_2^m+\eps)^\frac{1}{m}, 
	\qquad x\in\R^n,\; t>0. 
\end{equation}

We give an upper bound of $u_\eps$ uniformly for $\eps$. 

\begin{lemma}\label{lem:upap}
For each $0<\eps<1$, 
\[
	u_\eps(x,t) \leq (c_1^m |x|^{-m\lambda} +c_2^m+1 )^\frac{1}{m}
\]
for any $x\in \R^n\setminus\{0\}$ and $t>0$. 
\end{lemma}

\begin{proof}
Let $0<\eps<1$. Direct computations and the choice of $\lambda$ show that 
\begin{equation}\label{eq:inisup}
\begin{aligned}
	- \Delta u_{0,\eps}^m 
	&= c_1^m m\lambda \left( n - (2 + m\lambda) \frac{|x|^2}{|x|^2+\eps} \right) 
	(|x|^2+\eps)^{-\frac{1}{2}m\lambda -1} \\
	&\geq c_1^m m\lambda ( n - 2 - m\lambda)  
	(|x|^2+\eps)^{-\frac{1}{2}m\lambda -1} \geq 0. 
\end{aligned}
\end{equation}
Then, $u_{0,\eps}$ is a supersolution of \eqref{eq:fastap}. 
Since $u_\eps(\,\cdot\,, 0) = u_{0,\eps}$, 
the comparison principle for bounded solutions gives 
\[
	u_\eps(x,t) \leq u_{0,\eps}(x)
	= ( c_1^m (|x|^2+\eps)^{-\frac{1}{2}m \lambda} 
	+c_2^m +\eps )^\frac{1}{m}
	\leq (c_1^m |x|^{-m\lambda} +c_2^m +1)^\frac{1}{m}. 
\]
The lemma follows. 
\end{proof}

As a by-product of \eqref{eq:inisup}, 
we can see that $u_\eps$ is decreasing in $t$. 

\begin{lemma}\label{lem:timedec}
For each $0<\eps<1$, $\partial_t u_\eps \leq0$ in $\R^n\times[0,\infty)$. 
\end{lemma}

\begin{proof}
This lemma can be proved by the same argument as in \cite[Lemma 3.4]{VW11}. 
We give an outline. 
Define $z:= \partial_t u_\eps$. 
Remark that 
$z$ is smooth by the parabolic regularity theory. 
Then, \eqref{eq:inisup} implies $z(\,\cdot\,,0)\leq 0$. 
We differentiate $\partial_t u_\eps=\Delta u_\eps^m$ 
with respect to $t$ and obtain 
\[
	z_t = \frac{m}{u_\eps^{1-m}} \Delta z 
	+ \frac{2 m (m-1) \nabla u_\eps}{u_\eps^{2-m}} \cdot \nabla z 
	+ \frac{ m  (m-1) (m-2) |\nabla u_\eps|^2  
	+  m (m-1) u_\eps \Delta u_\eps }{u_\eps^{3-m}} z. 
\]
By \eqref{eq:roughes} and the parabolic regularity theory, 
we see that the coefficients are bounded by a constant depending on $\eps$. 
In addition, $z$ is bounded. 
By \eqref{eq:roughes} again, 
the coefficient of $\Delta z$ is bounded below by a positive constant 
also depending on $\eps$. 
Hence, the comparison principle for bounded solutions 
yields $z=\partial_t u_\eps \leq0$. 
\end{proof}

We also prove the radial symmetry and monotonicity of $u_\eps$.

\begin{lemma}\label{lem:ueprad}
$u_\eps$ is radially symmetric and nonincreasing in $|x|$ 
for each $t\in [0,\infty)$. 
\end{lemma}

\begin{proof}
This lemma follows from the same argument as in 
\cite[Subsection 5.7.1]{Vabook}. We give an outline. 
Since $u_{0,\eps}$ is radially symmetric, 
the rotationally invariance of $\Delta$ 
and the uniqueness of a solution of \eqref{eq:fastap} 
implies that $u_\eps(x,t)$ is radially symmetric for each $t\in[0,\infty)$. 

We regard $u_\eps(x,t)$ as a function $u_\eps(r,t)$ ($r:=|x|$). 
Set $\tilde z:=\partial_r u_\eps$. Then, $\tilde z(r,0)\leq 0$ for $r\geq0$. 
We differentiate 
$\partial_t u_\eps=\partial_r^2 u_\eps^m+ (n-1)r^{-1}\partial_r u_\eps^m$ 
with respect to $r$ and obtain 
\[
	\tilde z_t=\partial_r^2 (m u_\eps^{m-1} \tilde z)
	+ \frac{n-1}{r} \partial_r (m u_\eps^{m-1} \tilde z)
	-\frac{n-1}{r^2} m u_\eps^{m-1} \tilde z. 
\]
By \eqref{eq:roughes} and the comparison principle for bounded solutions, 
we obtain $\tilde z =\partial_r u_\eps\leq 0$. 
\end{proof}

We construct a solution of \eqref{eq:fast}. 

\begin{lemma}\label{lem:ex}
There exists a nonnegative solution  
$u\in X_\infty$ of \eqref{eq:fast}. 
Moreover, $u$ satisfies (i) and (ii) of Theorem \ref{th:blowdown}. 
\end{lemma}

\begin{proof}
Lemma \ref{lem:upap} and \eqref{eq:roughes} give, for each $\eps>0$, 
\[
	0< u_\eps(x,t) \leq (c_1^m |x|^{-m\lambda} +c_2^m+1 )^\frac{1}{m}, 
	\qquad 
	x\in \R^n\setminus\{0\},\; t>0. 
\]
The parabolic regularity theory and 
the standard diagonalization argument show that 
$u_\eps$ converges along subsequence $u_{\eps_i}$ to a nonnegative function 
$u$ in $C_\loc((\R^n\setminus \{0\})\times[0,\infty))$ and in 
$C^{2,1}_\loc((\R^n\setminus \{0\})\times(0,\infty))$ as $i\to\infty$. 
Then, by \eqref{eq:fastap} with $\eps=\eps_i$ and letting $i\to\infty$, 
we see that the nonnegative function $u$ satisfies \eqref{eq:fast}. 
The properties (i) and (ii) follow from 
Lemmas \ref{lem:ueprad} and \ref{lem:timedec}, respectively. 
\end{proof}

We note that, eventually, the nonnegative solution constructed in Lemma \ref{lem:ex} 
is a unique positive solution of \eqref{eq:fast}, 
see  Lemma \ref{lem:loap} for positivity 
and Theorem \ref{th:uni} for uniqueness.

\subsection{Uniqueness}\label{subsec:comuni}
We prove Theorem \ref{th:uni}. 
We remark that the following proof does not valid for $m\geq (n-2)/n$, 
see \eqref{eq:asuse} and \eqref{eq:asuse2} below.

\begin{proof}[Proof of Theorem \ref{th:uni}]
This proof is based on the argument of \cite[Theorem 2.3]{HP85}. 
Let $0<t<T\leq \infty$ 
and let $u_1, u_2\in X_T$ satisfy 
$u_1(\,\cdot\,,0)=u_2(\,\cdot\,,0)$ on $\R^n \setminus \{0\}$ 
and \eqref{eq:fasteq}. 
In this proof, we write $u(t)=u(x,t)$ when no confusion can arise. 
Set $\sign(f):=f/|f|$ for $f\neq 0$ and 
$\sign(f):=0$ for $f=0$.

We first test the equation 
by a nonnegative function in $C^\infty_0(\R^n\setminus\{0\})$. 
Since the support of such a function 
does not contain the singular point $x=0$, 
the proof from here to \eqref{eq:uuode} is 
the same as \cite{HP85}, see \cite[Proof of Theorem 2.3]{HP85} for details. 
We recall Kato's inequality 
$\sign(f) \Delta f \leq \Delta |f|$ 
for any domain $D\subset \R^n$ and 
$f\in L^1_\loc(D)$ with $\Delta f\in L^1_\loc(D)$. 
We also recall that 
$\partial_t |g| = \sign(g) \partial_t g$ 
for any interval $I\subset \R$ and 
$g\in L^1_\loc(I)$ with $\partial_t g \in L^1_\loc(I)$. 
Then, we can see that 
\[
\begin{aligned}
	0&=\sign(u_1-u_2) \partial_t(u_1-u_2)
	- \sign(u_1-u_2) \Delta (u_1|u_1|^{m-1}-u_2|u_2|^{m-1})  \\
	&\geq 
	\partial_t |u_1-u_2| - \Delta \left| u_1|u_1|^{m-1}-u_2|u_2|^{m-1}\right|, 
\end{aligned}
\]
and that, for any nonnegative function 
$\psi_0\in C^\infty_0(\R^n\setminus\{0\})$, 
\begin{align}
	&\notag
	\frac{d}{dt} \int_{\R^n} \psi_0 |u_1-u_2| dx 
	\leq 
	\int_{\R^n} \Delta \psi_0 
	\left| u_1 |u_1|^{m-1} -u_2|u_2|^{m-1}\right| dx, \\
	&\label{eq:uuint}
	\int_{\R^n} \psi_0 |u_1(t)-u_2(t)| dx 
	\leq 
	\int_0^t \int_{\R^n} 
	\Delta \psi_0 
	\left| u_1 |u_1|^{m-1} -u_2|u_2|^{m-1}\right| dxd\tau. 
\end{align}
By the H\"older inequality and 
$|a|a|^{\tilde m-1}-b|b|^{\tilde m-1}|\leq 2^{1-\tilde m}|a-b|^{\tilde m}$ 
($a,b\in\R$, $0<\tilde m<1$), we have 
\[
\begin{aligned}
	&\int_{\R^n} \Delta \psi_0 
	\left| u_1 |u_1|^{m-1} -u_2|u_2|^{m-1}\right| dx \\
	&\leq 
	2^{1-m}\int_{\R^n} |\Delta \psi_0| |u_1-u_2|^m dx \\
	&\leq 
	(2C[\psi_0])^{1-m} 
	\left( \int_{\R^n} \psi_0 |u_1-u_2| dx  \right)^m, 
	\qquad C[\psi_0] := 
	\int_{\R^n} |\Delta \psi_0|^\frac{1}{1-m} \psi_0^{-\frac{m}{1-m}} dx. 
\end{aligned}
\]
Thus, 
\begin{equation}\label{eq:uuode}
	\int_{\R^n} \psi_0 |u_1(t)-u_2(t)| dx
	\leq 2(1-m)^\frac{1}{1-m} C[\psi_0] t^\frac{1}{1-m}. 
\end{equation}

We choose $\tilde \phi\in C^\infty (\R^n)$ 
such that $0\leq \tilde \phi\leq 1$, $\tilde \phi=0$ if $|x|\leq 1/2$ 
and $\tilde \phi=1$ if $|x|\geq 1$.
Define $\phi:=\tilde \phi^k$ with $k>2/(1-m)$. 
For $0<\eps<1$, set $\phi_\eps(x):=\phi(x/\eps)$. 
Let $\tilde \psi \in C^\infty_0(\R^n)$ be a nonnegative function. 
Set $\psi:=\tilde \psi^k$ and $\psi_\eps:=\psi \phi_\eps$. 
Remark that $\psi_\eps\in C^\infty_0(\R^n\setminus\{0\})$. 
In this setting, we see that 
\[
	C[\psi] =
	\int_{\R^n} 
	\left|k(k-1) \tilde \psi^{k(1-m)-2} |\nabla \tilde\psi|^2 
	+ k\tilde \psi^{k(1-m)-1} \Delta \tilde \psi\right|^\frac{1}{1-m} dx <\infty, 
	\qquad C[\phi]<\infty. 
\]
By setting 
$\tilde C[\phi] 
:= \int_{\R^n} |\nabla \phi|^{1/(1-m)} \phi^{-m/(1-m)} dx$, 
we have $\tilde C[\phi] <\infty$. 
From straightforward estimates and 
the change of variables $y=x/\eps$, it follows that 
\[
\begin{aligned}
	C[\psi_\eps] &= 
	\int_{\R^n} | \phi_\eps \Delta \psi
	+ 2\nabla \phi_\eps\cdot \nabla \psi
	+\psi \Delta \phi_\eps |^\frac{1}{1-m} 
	\psi^{-\frac{m}{1-m}} \phi_\eps^{-\frac{m}{1-m}} dx \\
	&\leq 
	3^\frac{1}{1-m} C[\psi] 
	+
	C \int_{\R^n} 
	| \nabla \phi_\eps
	\cdot k\tilde \psi^{k(1-m)-1}\nabla \tilde \psi |^\frac{1}{1-m} 
	\phi_\eps^{-\frac{m}{1-m}} dx  
	+ C \|\psi\|_{L^\infty(\R^n)}  C[\phi_\eps] \\ 
	&\leq 
	3^\frac{1}{1-m} C[\psi] 
	+C \tilde C[\phi_\eps] + C C[\phi_\eps] \\ 
	&=
	3^\frac{1}{1-m} C[\psi] 
	+C \eps^{n-\frac{1}{1-m}} \tilde C[\phi] 
	+ C\eps^{n-\frac{2}{1-m}} C[\phi]. 
\end{aligned}
\]
Since $C[\phi],\tilde C[\phi]<\infty$ and $n-2/(1-m)>0$, 
we obtain
\begin{equation}\label{eq:asuse}
	\liminf_{\eps\to0} C[\psi_\eps] \leq 
	3^\frac{1}{1-m} C[\psi] 
	+C \liminf_{\eps\to0}( \eps^{n-\frac{1}{1-m}} 
	+ \eps^{n-\frac{2}{1-m}} )
	= 3^\frac{1}{1-m} C[\psi]. 
\end{equation}
Then, by substituting $\psi_0=\psi_\eps$ into \eqref{eq:uuode} and 
applying Fatou's lemma, we obtain 
\begin{equation}\label{eq:ulocint}
	\int_{\R^n} \psi |u_1(t)-u_2(t)| dx 
	\leq 2(3(1-m))^\frac{1}{1-m} C[\psi] t^\frac{1}{1-m}, 
	\qquad C[\psi]<\infty 
\end{equation}
for any 
nonnegative function $\psi \in C^\infty_0(\R^n)$ 
that can be written in the form $\psi=\tilde \psi^k$ 
with $\tilde \psi\in C^\infty_0(\R^n)$ and $k>2/(1-m)$. 
In particular, $u_1-u_2\in L^1(K\times(0,t))$ 
for any compact subset $K$ of $\R^n$. 
This together with $|a|a|^{m-1}-b|b|^{m-1}|\leq |a-b|^m$  
and the H\"older inequality also shows that 
$|u_1-u_2|^m \in L^1(K\times(0,t))$ 
and $u_1 |u_1|^{m-1} -u_2|u_2|^{m-1} \in 
L^1(K\times(0,t))$.

Let $\varphi\in C^\infty_0(\R^n)$ be a nonnegative function. 
By substituting $\psi_0=\varphi \phi_\eps$ into \eqref{eq:uuint} 
and applying Fatou's lemma, we have 
\[
	\int_{\R^n} \varphi |u_1(t)-u_2(t)| dx 
	\leq \liminf_{\eps\to0} (I_{1,\eps}+I_{2,\eps}), 
\]
where 
\[
\begin{aligned}
	& I_{1,\eps}:= 
	\int_0^t \int_{\R^n} \phi_\eps \Delta \varphi
	\left|u_1|u_1|^{m-1}- u_2|u_2|^{m-1}\right| dx d\tau, \\
	& I_{2,\eps}:= 
	\int_0^t \int_{\R^n} (2\nabla \varphi\cdot \nabla \phi_\eps 
	+ \varphi \Delta \phi_\eps) 
	\left|u_1|u_1|^{m-1}- u_2|u_2|^{m-1}\right| dx d\tau. 
\end{aligned}
\]
Since $u_1 |u_1|^{m-1} -u_2|u_2|^{m-1} \in 
L^1(K\times(0,t))$ for $K$ compact, 
Lebesgue's dominated convergence theorem gives 
$I_{1,\eps}\to 
\int_0^t \int_{\R^n} \Delta \varphi 
|u_1 |u_1|^{m-1} -u_2|u_2|^{m-1}| dx d\tau$ as $\eps\to0$. 
By $|a|a|^{m-1}-b|b|^{m-1}|\leq |a-b|^m$ 
and the H\"older inequality, we have 
\[
\begin{aligned}
		|I_{2,\eps}| &\leq 
	\int_0^t \int_{\R^n} 
	( 2|\nabla \varphi|  |\nabla \phi_\eps| + \varphi |\Delta \phi_\eps| )
	|u_1-u_2|^m dx d\tau \\
	&\leq 
	\left( \int_0^t \int_{\R^n} 
	 2^\frac{1}{m} |\nabla \varphi|^\frac{1}{m} 
	|u_1-u_2|  dx d\tau \right)^m 
	\left( \int_0^t \int_{B(0;\eps)} |\nabla \phi_\eps|^\frac{1}{1-m} dxd\tau 
	\right)^{1-m} \\
	&\quad + 
	\left( \int_0^t \int_{\R^n} 
	\varphi^\frac{1}{m} |u_1-u_2|  dx d\tau \right)^m 
	\left( \int_0^t \int_{B(0;\eps)} |\Delta \phi_\eps|^\frac{1}{1-m} dxd\tau 
	\right)^{1-m}. 
\end{aligned}
\]
From $|u_1-u_2| \in L^1(K\times(0,t))$ for $K$ compact, 
the change of variables $y=x/\eps$ 
and $n-2/(1-m)>0$, 
it follows that 
\begin{equation}\label{eq:asuse2}
	|I_{2,\eps}| \leq 
	C ( \eps^{(n-\frac{1}{1-m})(1-m)} 
	+ \eps^{(n-\frac{2}{1-m})(1-m)} ) 
	\to 0 \qquad \mbox{ as }\eps\to0. 
\end{equation}
Therefore, 
for any nonnegative function $\varphi \in C^\infty_0(\R^n)$, 
we obtain 
\[
	\int_{\R^n} \varphi |u_1(t)-u_2(t)| dx 
	\leq 
	\int_0^t \int_{\R^n} \Delta \varphi 
	\left|u_1 |u_1|^{m-1} -u_2|u_2|^{m-1}\right| dx d\tau. 
\]

The rest of the proof is the same as 
the latter part of \cite[Theorem 2.3]{HP85}. 
Set 
\[
	w(x,t):=\int_0^t \left|u_1(x,\tau) |u_1(x,\tau)|^{m-1} 
	- u_2(x,\tau)|u_2(x,\tau)|^{m-1}\right| d\tau. 
\]
By $\Delta \varphi |u_1 |u_1|^{m-1} -u_2|u_2|^{m-1}| 
\in L^1(\R^n\times(0,t))$, 
Fubini's theorem gives 
\[
	\int_{\R^n} \varphi |u_1(t)-u_2(t)| dx 
	\leq 
	\int_{\R^n} w(x,t) \Delta \varphi(x)  dx. 
\]
Then, $\int_{\R^n} (\Delta \varphi) w  dx\geq0$, 
and so $-\Delta w\leq 0$ in $\cD'(\R^n)$. 
Hence, the following mean value inequality for subharmonic functions holds 
\[
	w(\xi,t) \leq \frac{1}{\omega_n R^n}\int_{B(\xi;R)} w(x,t) dx
	:= I_R 
\]
for $\xi\in \R^n$ and $R>0$, 
where $\omega_n$ is the volume of a unit ball. 
Thus, $u_1\equiv u_2$ will be proved once 
we prove $I_R\to0$ as $R\to\infty$.

We choose $\tilde \psi_1\in C^\infty_0(\R^n)$ such that 
$0\leq \tilde \psi_1\leq 1$, $\tilde \psi_1=0$ if $|x-\xi|\geq 2$ 
and $\tilde \psi_1=1$ if $|x-\xi|\leq 1$. 
Let $\psi_1:=\tilde \psi_1^k$ with $k>2/(1-m)$. 
For $R>1$, set $\psi_R(x):=\psi_1(x/R)$. 
Then, by the H\"older inequality and \eqref{eq:ulocint} with $\psi=\psi_R$, 
we have 
\[
\begin{aligned}
	I_R &\leq 
	\frac{2^{1-m}}{\omega_n R^n}  \int_0^t
	\int_{B(\xi;R)} |u_1-u_2|^m dx d\tau \\
	&\leq 
	C R^{-n+n(1-m)} \int_0^t
	\left( \int_{B(\xi;R)} |u_1-u_2| dx \right)^m  d\tau \\
	&\leq 
	C R^{-n+n(1-m)} \int_0^t
	\left( \int_{\R^n} \psi_R |u_1-u_2| dx \right)^m  d\tau 
	\leq 
	C R^{-nm} C[\psi_R]^m t^\frac{1}{1-m}. 
\end{aligned}
\]
Since $C[\psi_1]<\infty$ and 
\[
	C[\psi_R]
	= 
	\int_{B(\xi;2R)} R^{-\frac{2}{1-m}} 
	\left|(\Delta \psi_1) \left( \frac{x}{R} \right) \right|^\frac{1}{1-m} 
	\psi_1\left( \frac{x}{R} \right)^{-\frac{m}{1-m}} dx 
	= 
	C[\psi_1] R^{n-\frac{2}{1-m}}, 
\]
we obtain 
\[
	I_R 
	\leq 
	C R^{-\frac{2m}{1-m}} t^\frac{1}{1-m} \to 0 
	\qquad \mbox{ as }R\to\infty. 
\]
As stated before, this shows $u_1\equiv u_2$. 
The proof is complete. 
\end{proof}

The uniqueness part in Theorem \ref{th:blowdown} 
immediately follows from Theorem \ref{th:uni}.

\subsection{Estimates of the solution}\label{subsec:estisol}
To estimate the unique solution of \eqref{eq:fast}, 
we construct a classical supersolution $\overline{u}$ 
and a continuous weak subsolution $\underline{u}$. 
Here we call $\overline{u}>0$ a classical supersolution if 
$\overline{u}$ belongs to $X_\infty$ 
and satisfies $\overline{u}(\,\cdot\,,0)\geq u_0$ and 
\[
	\partial_t \overline{u} \leq \Delta \overline{u}^m 
	\qquad \mbox{ in }(\R^n\setminus\{0\})\times (0,\infty). 
\]
We call $\underline{u}>0$ 
a continuous weak subsolution 
if $\underline{u}$ belongs to $C((\R^n\setminus\{0\})\times[0,\infty))$ 
and satisfies $\underline{u}(\,\cdot\,,0)\leq u_0$ and 
\begin{equation}\label{eq:cwsub}
	\int_{\R^n}\underline{u}(x,t) \varphi(x,t) dx - 
	\int_{\R^n}u_0(x) \varphi(x,0) dx 
	\leq \int_0^t \int_{\R^n} 
	(\underline{u} \varphi_t + \underline{u}^m \Delta \varphi) dxd\tau 
\end{equation}
for any $0<t<\infty$, 
where $\varphi\in C^\infty((\R^n\setminus\{0\})\times[0,\infty))$ 
is a nonnegative function such that 
$\supp \varphi(\,\cdot\,,t)$ is a compact subset of 
$\R^n\setminus \{0\}$ for any $t\in [0,\infty)$. 
We note that the existence of $\overline{u}$ and $\underline{u}$ 
guarantees the existence of a positive solution 
$\tilde u$ of \eqref{eq:fast} 
satisfying $\underline{u}\leq \tilde u \leq \overline{u}$ 
in $(\R^n\setminus\{0\})\times [0,\infty)$. 
Then, by the uniqueness of a solution of \eqref{eq:fast}, 
the unique solution $u$ of \eqref{eq:fast} is also estimated as  
$\underline{u}\leq u \leq \overline{u}$ 
in $(\R^n\setminus\{0\})\times [0,\infty)$.

We first give a supersolution $\overline{u}$. 
Set 
\[
	\overline{u}(x,t):=u_{0}(x)= (c_1^m |x|^{-m\lambda} +c_2^m )^\frac{1}{m}. 
\]

\begin{lemma}\label{lem:upsol}
$\overline{u}$ is a classical supersolution 
in $(\R^n\setminus\{0\})\times [0,\infty)$. 
\end{lemma}

\begin{proof}
By similar computations to \eqref{eq:inisup}, we have 
\[
	\partial_t \overline{u} - \Delta \overline{u}^m 
	= 0+c_1^m m\lambda ( n - 2 - m\lambda ) 
	|x|^{-m\lambda -2} \geq 0. 
\]
Since $\overline{u}(\,\cdot\,,0)= u_0$, the lemma follows. 
\end{proof}

Let us next construct subsolutions. 
We prepare auxiliary constants and functions as follows. 
Let $\nu,\mu,\delta>0$ satisfy 
\[
	\lambda-\frac{2}{m}\left( 
	\frac{\lambda}{2}(1-m)-1 \right) 
	<\nu<\lambda<\frac{n-2}{m}<\mu<\frac{2n-2}{m}, \qquad 
	\frac{\mu-\lambda}{\mu-\nu} <\delta <1. 
\]
For $t\geq0$, set 
\[
\begin{aligned}
	&a(t):=A e^t, 
	&&\rho(t)
	:= \delta^\frac{1}{m(\lambda-\nu)} a(t)^{-\frac{1}{m(\lambda-\nu)}}, \\
	&b(t):= (1-\delta) \rho(t)^{m(\mu-\lambda)}, 
	&&\sigma(t):= (c_1/c_2)^\frac{1}{\mu} b(t)^\frac{1}{m\mu}, 
\end{aligned}
\]
where $\sigma$ is defined for $c_2>0$ and $A>1$ is a constant satisfying 
\begin{align}
	&\label{eq:Alarge1}
	(c_1/c_2)^\frac{1}{\mu}
	(1-\delta)^\frac{1}{m\mu} 
	\delta^{ \frac{1}{m(\lambda-\nu)} \frac{\mu-\lambda}{\mu}} 
	A^{\frac{1}{m(\lambda-\nu)} \frac{\lambda}{\mu} } 
	- \delta^\frac{1}{m(\lambda-\nu)}  >0 \qquad \mbox{ if }c_2>0, \\
	& \label{eq:Alarge2}
	c_1^{1-m} A (1-\delta)^{\frac{1}{m}-1} > nm^2 \lambda. 
\end{align}
Remark that $a(t)>1$, $\rho(t)<1$ and $b(t)<1$ for $t\geq0$ 
and also remark that $a(t)\uparrow\infty$, 
$\rho(t)\downarrow 0$ and $b(t)\downarrow 0$ 
as $t\uparrow \infty$. 

For $r>0$, define 
\[
\left\{ 
\begin{aligned}
	&w_{\inn}(r,t):= 
	c_1 \left[ r^{-m\lambda} 
	- a(t) r^{-m\nu} \right]_+^\frac{1}{m}, \\
	&w_{\out}(r,t):= 
	c_1 b(t)^\frac{1}{m} r^{-\mu}, 
\end{aligned}
\right.
\]
where $[\,\cdot\,]_+$ is the positive part. 
We set 
\[
\begin{aligned}
	&\mbox{for }c_2>0, \qquad 
	\underline{u}(x,t):=
	\left\{
	\begin{aligned}
	&w_{\inn}(|x|,t) &&\mbox{ for } 0<|x|\leq \rho(t), \\
	&w_{\out}(|x|,t) &&\mbox{ for } \rho(t)< |x| \leq \sigma(t), \\
	&c_2 &&\mbox{ for } |x| > 
	\sigma(t), 
	\end{aligned}
	\right. \\
	&\mbox{for }c_2=0, \qquad 
	\underline{u}(x,t):=
	\left\{
	\begin{aligned}
	&w_{\inn}(|x|,t) &&\mbox{ for } 0<|x|\leq \rho(t), \\
	&w_{\out}(|x|,t) &&\mbox{ for } |x|> \rho(t). 
	\end{aligned}
	\right.
\end{aligned}
\]
We can see that $w_{\inn}>0$ for $r\leq \rho(t)$ 
and $w_{\inn}=w_{\out}$ for $r=\rho(t)$. 
In the case $c_2>0$, we also have 
$w_{\out}=c_2$ for $r=\sigma(t)$. 
Then, $\underline{u}>0$ and 
$\underline{u}\in C((\R^n\setminus\{0\})\times[0,\infty))$. 
Moreover, $\rho(t)< \sigma(t)$ for $t\geq0$ 
if $A$ satisfies \eqref{eq:Alarge1}. 
Indeed, 
\[
\begin{aligned}
	&\sigma(t) - \rho(t) \\
	&= (c_1/c_2)^\frac{1}{\mu} (1-\delta)^\frac{1}{m\mu} 
	\delta^{ \frac{1}{m(\lambda-\nu)} \frac{\mu-\lambda}{\mu}} 
	a(t)^{-\frac{1}{m(\lambda-\nu)} \frac{\mu-\lambda}{\mu} } 
	- \delta^\frac{1}{m(\lambda-\nu)} a(t)^{-\frac{1}{m(\lambda-\nu)}} \\
	&= 
	a(t)^{-\frac{1}{m(\lambda-\nu)}} 
	( (c_1/c_2)^\frac{1}{\mu}
	(1-\delta)^\frac{1}{m\mu} 
	\delta^{ \frac{1}{m(\lambda-\nu)} \frac{\mu-\lambda}{\mu}} 
	a(t)^{\frac{1}{m(\lambda-\nu)} \frac{\lambda}{\mu} } 
	- \delta^\frac{1}{m(\lambda-\nu)}  
	) \\
	&\geq 
	a(t)^{-\frac{2}{m(\lambda-\nu)}} 
	( 
	(c_1/c_2)^\frac{1}{\mu}
	(1-\delta)^\frac{1}{m\mu} 
	\delta^{ \frac{1}{m(\lambda-\nu)} \frac{\mu-\lambda}{\mu}} 
	A^{\frac{1}{m(\lambda-\nu)} \frac{\lambda}{\mu} } 
	- \delta^\frac{1}{m(\lambda-\nu)}  
	) > 0. 
\end{aligned}
\]
Note that \eqref{eq:Alarge2} will be used for proving that 
$\underline{u}$ is a subsolution for $0< |x| < \rho(t)$.

\begin{lemma}\label{lem:loap}
$\underline{u}$ is a continuous weak subsolution in 
$(\R^n\setminus\{0\})\times [0,\infty)$. 
\end{lemma}

\begin{proof}
We only prove the case $c_2>0$. 
The proof for $c_2=0$ is the same. 
We first prove that 
$\partial_t \underline{u}\leq \Delta \underline{u}^m$ except for 
$|x|= \rho(t)$ or $|x|=\sigma(t)$. 
For $0< |x| < \rho(t)$ and $t>0$, 
\[
\begin{aligned}
	\partial_t \underline{u} - \Delta \underline{u}^m 
	&= -\frac{c_1}{m} a'(t) 
	( |x|^{-m\lambda} - a(t) |x|^{-m\nu} )^{\frac{1}{m}-1} |x|^{-m\nu}  \\
	&\quad 
	-c_1^m m\lambda(m\lambda+2-n) |x|^{-m\lambda -2 } 
	+c_1^m m\nu(m\nu+2-n) a(t) |x|^{-m\nu-2}. 
\end{aligned}
\]
By $|x| < \rho(t)$, we have 
$|x|^{-m\lambda} - a(t) |x|^{-m\nu} \geq (1-\delta)|x|^{-m\lambda}$. 
Then, by $a'(t)\geq A$, $m\nu+2-n<0$, $\lambda(1-m)-2 - m(\lambda-\nu)>0$, 
$|x| < \rho(t)<1$ and \eqref{eq:Alarge2}, we obtain
\[
\begin{aligned}
	\partial_t \underline{u} - \Delta \underline{u}^m 
	&\leq - \frac{c_1 A}{m} 
	(1-\delta)^{\frac{1}{m}-1}
	|x|^{ - \lambda(1-m) - m\nu } 
	+c_1^m nm\lambda |x|^{-m\lambda-2} \\
	&= m^{-1} c_1^m |x|^{-\lambda(1-m)-m\nu} 
	( - c_1^{1-m}A 
	(1-\delta)^{\frac{1}{m}-1}
	+nm^2 \lambda |x|^{\lambda(1-m)- 2 - m(\lambda-\nu)} ) \\
	&\leq m^{-1} c_1^m |x|^{-\lambda(1-m)-m\nu} 
	( nm^2 \lambda - c_1^{1-m} A (1-\delta)^{\frac{1}{m}-1}  )
	\leq 0. 
	\end{aligned}
\]
For $\rho(t)<|x|< \sigma(t)$ and $t>0$, 
by $b'(t)\leq 0$ and the choice of $\mu$, 
\[
	\partial_t \underline{u} - \Delta \underline{u}^m 
	= \frac{c_1}{m} b(t)^{\frac{1}{m}-1} b'(t) |x|^{-\mu} 
	- c_1^mm\mu ( m\mu +2-n) b(t) |x|^{-m\mu -2}\leq 0. 
\]
For $|x|>\sigma(t)$ and $t>0$, we have 
$\partial_t \underline{u} - \Delta \underline{u}^m = 0\leq 0$.

We next check the following appropriate matching conditions 
\[
\left\{ 
\begin{aligned}
	&\partial_r w_{\inn}^m < \partial_r w_{\out}^m 
	&&\mbox{ if }r=\rho(t), \\
	&\partial_r w_{\out}^m < \partial_r c_2^m 
	&&\mbox{ if }r=\sigma(t), 
\end{aligned}
\right.
\]
for $t\geq0$, where $r=|x|$. 
For $r=\rho(t)$, 
\[
\begin{aligned}
	\partial_r w_{\out}^m - \partial_r w_{\inn}^m 
	&= 
	c_1^m m ( -\mu b(t) \rho(t)^{-m\mu-1} 
	+\lambda \rho(t)^{-m\lambda -1}
	-\nu a(t) \rho(t)^{-m\nu -1} ) \\
	&= 
	c_1^m m \rho(t)^{-m\lambda-1} (\mu-\nu)
	\left( \delta - \frac{\mu-\lambda}{\mu-\nu} \right) >0
\end{aligned}
\]
by the choice of $\delta$. The case $r=\sigma(t)$ is clear.

Let $\varphi\in C^\infty((\R^n\setminus\{0\})\times[0,\infty))$ 
be a nonnegative function such that 
$\supp \varphi(\,\cdot\,,t)$ is a compact subset of 
$\R^n\setminus \{0\}$ for any $t\in [0,\infty)$. 
From integration by parts, it follows that 
\[
\begin{aligned}
	\int_0^t \int_{\R^n} 
	(\underline{u} \partial_t \varphi +\underline{u}^m \Delta \varphi) dxd\tau 
	&=
	\int_{\R^n} \underline{u}(x,t)\varphi(x,t) dx 
	- \int_{\R^n} \underline{u}(x,0)\varphi(x,0) dx \\
	&\quad 
	- \int_0^t \int_{
	\{|x|<\rho\} \cup \{\rho<|x|<\sigma\} 
	\cup \{|x|>\sigma\}  } 
	(\partial_t \underline{u}-\Delta \underline{u}^m) \varphi dxd\tau \\
	&\quad 
	- \int_0^t \int_{\{|x|= \rho\}} 
	( (\partial_r w_\inn^m)(|x|,\tau) 
	- (\partial_r w_\out^m)(|x|,\tau) ) \varphi dS d\tau \\
	&\quad 
	-\int_0^t \int_{\{|x|= \sigma\}} 
	( (\partial_r w_\out^m)(|x|,\tau) - \partial_r c_2^m ) 
	\varphi dS d\tau 
\end{aligned}
\]
for any $0<t<\infty$. 
Hence, by 
$\partial_t \underline{u}\leq \Delta \underline{u}^m$ except for 
$|x|= \rho(t)$ or $|x|=\sigma(t)$ 
and by the matching condition, 
$\underline{u}$ satisfies \eqref{eq:cwsub} for any $0<t<\infty$.

Finally, we check $\underline{u}(\,\cdot\,,0)\leq u_0$ on $\R^n\setminus\{0\}$. 
For $0< |x| \leq \rho(0)$, 
\[
\begin{aligned}
	\underline{u}(x,0)  
	= c_1 ( |x|^{-m\lambda} - a(0) |x|^{-m\nu} )^\frac{1}{m} 
	\leq 
	( c_1^m |x|^{-m \lambda} + c_2^m )^\frac{1}{m}
	=u_0(x). 
\end{aligned}
\]
For $\rho(0)< |x| \leq \sigma(0)$, 
\[
\begin{aligned}
	u_0(x) - \underline{u}(x,0) 
	&= 
	( c_1^m |x|^{-m \lambda} + c_2^m )^\frac{1}{m}
	- c_1 (1-\delta)^\frac{1}{m} \rho(0)^{\mu-\lambda} 
	|x|^{-\mu} \\
	&\geq 
	c_1 |x|^{-\mu}
	(|x|^{\mu-\lambda} - (1-\delta)^\frac{1}{m} \rho(0)^{\mu-\lambda} ) \\
	&\geq 
	c_1 |x|^{-\mu} \rho(0)^{\mu-\lambda}
	( 1- (1-\delta)^\frac{1}{m} )  \geq 0. 
\end{aligned}
\]
The case $|x| \geq \sigma(0)$ is clear. 
The proof is complete. 
\end{proof}

By Lemmas \ref{lem:upsol} and \ref{lem:loap},
as noted in the first part of this subsection, 
the unique solution $u$ of \eqref{eq:fast} satisfies 
$\underline{u}\leq u\leq \overline{u}$. 
In particular, $u$ is positive. 
Hence, $u$ satisfies the properties (iii) and (v) of Theorem \ref{th:blowdown}. 
We prove (iv) in the following lemma.

\begin{lemma}
$u$ satisfies (iv) of Theorem \ref{th:blowdown}. 
\end{lemma}

\begin{proof}
The proof of this lemma is based on \cite{Ku19,Fu19}. 
In particular, the idea of computing \eqref{eq:FTay} below 
is due to \cite[Section 1]{Fu19}. 

Since $u$ is radially symmetric by the property (i) of Theorem \ref{th:blowdown}, 
we regard $u(x,t)$ as a function $u(r,t)$ ($r=|x|$). 
Thus, we prove that 
$r^{\lambda+1} \partial_r u(r,t)\to  -\lambda c_1$ 
locally uniformly for $t\in[0,\infty)$ as $r\to0$.
In this proof, we abbreviate the variable $t$ when no confusion can arise. 
For instance, $u(r,t)$ is written as $u(r)$. 
We also write $\eta=u_0(r)$. 
Note that $r\to0$ is equivalent to $\eta\to\infty$. 
By $u_0'(r)<0$ for $r>0$, 
the initial data $u_0$ has the inverse function $u_0^{-1}$. 
Set $\psi(\eta):=u_0^{-1}(\eta)$ and $F(\eta,t):=u(\psi(\eta),t)$. 
Remark that $r=\psi(\eta)$. 
Let $0<\eps<1/2$. 
From the Taylor expansion, it follows that 
\begin{equation}\label{eq:FTay}
\left\{ 
\begin{aligned}
	&\frac{F((1+\eps)\eta) - F(\eta)}{\eps \eta} - F'(\eta)
	= \eps\eta \int_0^1 F''((1+\theta\eps)\eta) (1-\theta) d\theta, \\
	&\frac{F(\eta) - F((1-\eps)\eta)}{\eps\eta} - F'(\eta) 
	= - \eps\eta \int_0^1 F''((1-\theta\eps)\eta) (1-\theta) d\theta. 
\end{aligned}
\right.
\end{equation}
The property (iii) of Theorem \ref{th:blowdown} 
gives $F(\eta)/\eta = u(r)/u_0(r) \to 1$ 
locally uniformly for $t\in[0,\infty)$
as $\eta\to\infty$. 
Thus, locally uniformly for $t\in[0,\infty)$ as $\eta\to\infty$, 
\begin{equation}\label{eq:Fdiflim}
\left\{ 
\begin{aligned}
	&\frac{F((1+\eps)\eta) - F(\eta)}{\eps \eta}
	= \left( \frac{F((1+\eps)\eta)}{(1+\eps)\eta} (1+\eps) 
	- \frac{F(\eta)}{\eta} \right) \frac{1}{\eps} 
	\to 1, \\
	&\frac{F(\eta) - F((1-\eps)\eta)}{\eps \eta}
	= \left( \frac{F(\eta)}{\eta} 
	- \frac{F((1-\eps)\eta)}{(1-\eps)\eta}(1-\eps)
	\right) \frac{1}{\eps}
	\to 1. 
\end{aligned}
\right. 
\end{equation}

Let $0<T<\infty$. 
We claim that there exist constants $C_T>0$ and $\eta_T>1$ 
depending on $T$ such that 
\begin{equation}\label{eq:etaF}
	\eta \partial_\eta^2 F(\eta,t) \leq C_T
	\qquad \mbox{ for }\eta>\eta_T, \; t\in [0,T). 
\end{equation}
By (iii) of Theorem \ref{th:blowdown}, 
$u_0'<0$ and $u_0''> 0$,  we have 
\begin{equation}\label{EFTPT}
	\eta F''(\eta) = 
	\frac{(u'' u_0' - u' u_0'')u_0}{(u_0')^3}
	\leq 
	\frac{u_0 u''}{(u_0')^2}, 
\end{equation}
where the prime symbols on $u$ mean partial derivatives 
with respect to $r$. 
Recall that $\partial_t u\leq 0$ by Lemma \ref{lem:ex}. 
Then, we have $(u^m)''+(n-1)r^{-1}(u^m)'\leq0$. 
Integrating this inequality over $(r,1)$, we have 
\[
\begin{aligned}
	(u^m)'(r) &\geq 
	(u^m)_r(1) + \int_r^1 \frac{n-1}{\xi} (u^m)' (\xi) d\xi \\
	&=
	(u^m)_r(1) + (n-1)u^m(1) -\frac{n-1}{r}u^m(r) 
	+\int_r^1 \frac{n-1}{\xi^2} u^m(\xi) d\xi \\
	&\geq 
	(u^m)_r(1)  -\frac{n-1}{r}u^m(r). 
\end{aligned}
\]
Then, by (iii) of Theorem \ref{th:blowdown}, 
there exist constants $C_T>0$ and $0<r_T<1$ such that 
$\partial_r u^m(r,t) \geq - C_T r^{-1} u^m(r,t)$ 
for $0<r<r_T$ and $t\in [0,T)$. 
Thus, 
\[
	0\geq \partial_r u(r,t) \geq - C_T r^{-\lambda-1} 
	\qquad \mbox{ for }0<r<r_T, \; t\in [0,T). 
\]
This together with $(u^m)''+(n-1)r^{-1}(u^m)'\leq0$ 
implies that 
\[
\begin{aligned}
	u''(r) \leq (1-m) u(r)^{-1} (u'(r))^2 -(n-1) r^{-1} u'(r) 
	\leq 
	C_T r^{-\lambda-2} 
\end{aligned}
\]
for $0<r<r_T$ and $t\in [0,T)$ by increasing $C_T$ 
and decreasing $r_T$ if necessary. 
Hence, by \eqref{EFTPT} and the definition of $u_0$, we obtain \eqref{eq:etaF}.

From \eqref{eq:etaF} and $0<\eps<1/2$, it follows that 
\[
\begin{aligned}
	&\eps\eta \int_0^1 F''((1+\theta\eps)\eta) (1-\theta) d\theta
	\leq 
	C_T \eps \int_0^1 \frac{1-\theta}{1+\theta\eps} d\theta 
	\leq C_T \eps, \\
	&- \eps\eta \int_0^1 F''((1-\theta\eps)\eta) (1-\theta) d\theta 
	\geq 
	- C_T \eps \int_0^1 \frac{1-\theta}{1-\theta\eps} d\theta 
	\geq -2C_T \eps. 
\end{aligned}
\]
These inequalities together with \eqref{eq:FTay} and \eqref{eq:Fdiflim}
show that 
\[
	1 - C_T \eps \leq \liminf_{\eta\to\infty} F'(\eta) 
	\leq \limsup_{\eta\to\infty} F'(\eta) \leq 1 + 2 C_T \eps. 
\] 
Letting $\eps\to 0$ gives $\partial_\eta F(\eta,t) \to 1$ 
locally uniformly for $t\in[0,\infty)$ as $\eta\to\infty$. 
Since $F'(\eta)=u'(\psi(\eta))\psi'(\eta)=u'(r)/u_0'(r)$ and 
$\eta\to\infty$ is equivalent to $r\to0$, 
we obtain $\partial_r u(r,t)/u_0'(r)\to 1$ 
locally uniformly for $t\in[0,\infty)$ as $r\to0$. 
This completes the proof. 
\end{proof}

\subsection{Convergence of the solution}\label{subsec:conv}

Finally, we show the convergence of the solution. 

\begin{lemma}
The unique solution $u$ of \eqref{eq:fast} 
satisfies the property (vi) of Theorem \ref{th:blowdown}. 
\end{lemma}

\begin{proof}
The following argument is due to \cite[Lemma 53.10]{QSbook}. 
Recall that $u$ satisfies $\partial_t u\leq 0$ in $\R^n\times(0,\infty)$ 
by the property (ii) of Theorem \ref{th:blowdown}. 
This together with $u>0$ shows that 
$v(x):=\lim_{t\to\infty} u(x,t)$ exists for each $x\in \R^n\setminus \{0\}$. 
On the other hand, let $v_i(x,t):=u(x,t+i)$. 
By Lemma \ref{lem:upsol}, we have 
$0<v_i(x,t) \leq (c_1^m |x|^{-m\lambda} +c_2^m)^{1/m}$ 
for $x\in \R^n\setminus\{0\}$ and $t>0$. 
Then, the parabolic regularity theory implies that 
$v_i$ converges along subsequence $v_{i'}$ to a function $w=w(x,t)$ in 
$C^{2,1}_\loc((\R^n\setminus\{0\})\times(0,\infty))$ as $i'\to\infty$. 
Note that $w$ satisfies $\partial_t w=\Delta w^m$ in 
$(\R^n\setminus\{0\})\times(0,\infty)$. 
By $v(x)=\lim_{t\to\infty} u(x,t)$, we have $v=w$. 
Thus, $u(\,\cdot\,,t) \to v$ in $C^2_\loc(\R^n\setminus\{0\})$ as $t\to\infty$. 
Moreover, $v$ satisfies $\Delta v^m=0$ in $\R^n\setminus\{0\}$.

We prove $v\equiv c_2$ on $\R^n\setminus\{0\}$. 
By Lemmas \ref{lem:upsol} and \ref{lem:loap}, we have 
\begin{equation}\label{eq:vbdd}
	c_2\leq v(x) \leq (c_1^m |x|^{-m\lambda} +c_2^m)^\frac{1}{m}, 
	\qquad x\in \R^n\setminus\{0\}. 
\end{equation}
In particular, by $\lambda<(n-2)/m$, we have 
$v(x)^m=o(|x|^{2-n})$ as $x\to0$. 
Then, $x=0$ is the removable singularity of $v^m$ 
(see \cite[Theorem 2.69]{Fobook} 
for the removability of the singularity in harmonic functions), 
and so $v^m$ can be extended as a harmonic function $\tilde v^m \in C^2(\R^n)$ 
satisfying $v^m\equiv \tilde v^m$ on $\R^n\setminus\{0\}$. 
By \eqref{eq:vbdd}, $\tilde v^m$ is a bounded harmonic function on $\R^n$. 
Hence, Liouville's theorem guarantees that $\tilde v^m$ is constant. 
By \eqref{eq:vbdd} again, we obtain $v \equiv \tilde v\equiv c_2$ 
on $\R^n\setminus\{0\}$. 
Hence, $u(\,\cdot\,,t) \to c_2$ in $C^2_\loc(\R^n\setminus\{0\})$ as $t\to\infty$. 
\end{proof}

The proof of Theorem \ref{th:blowdown} is complete.

\section{From analysis to geometry}\label{FATG}
In this section, we interpret the properties of the solution 
given in Section \ref{sec:exfas} into 
the ones of Riemannian metrics on $M=\mathbb{R}^{n}\setminus \{0\}$ and prove the main theorem, Theorem \ref{mainthm}. 

As explained in Subsection \ref{Yamabe-fast}, 
once we found the solution $u_{t}=u(\,\cdot\,,t)$ of \eqref{eq:fast}, 
then 
$g_{t}:=v_{t}^{4/(n-2)}g_{\mathbb{R}^n}$ 
becomes the solution of the Yamabe flow \eqref{y-flow}, where $v_{t}$ is defined by \eqref{uv}. 
More directly, $g_{t}:=u_{s}^{4/(n+2)}g_{\mathbb{R}^n}$ with $s=(n-2)t/((n-1)(n+2))$. 

Throughout this section, we fix $c_{1}, c_{2}>0$ and $(n+2)/2<\lambda<n+2$. 
We remark that we do not allow $c_{2}=0$ even it is allowed in Theorem \ref{th:blowdown}. 
Let $u_{t}=u(\,\cdot\,,t)$ be the unique long time solution 
guaranteed by  Theorem \ref{th:blowdown} with $m=m_{c}$. 
At this moment the first sentence of (4) of Theorem \ref{mainthm}, that is, 
the existence of a unique long time solution $g_{t}$ ($t\in[0,\infty)$) of the Yamabe flow \eqref{y-flow} starting from $g_{0}$ 
(defined in \eqref{initialg0}), has been already proved. 
We remark that the subscript $t$ in $u_{t}$ or $g_{t}$ are usually dropped to simplify the notation. 

Before proving the properties in Theorem \ref{mainthm}, 
it is better to fix the notions of 
``asymptotically Euclidean'' and ``asymptotically conical''. 
To do that, let $M$ be a general manifold (not restricted to $\mathbb{R}^{n}\setminus \{0\}$) with $\dim M=n$ and without boundary. 

\begin{definition}\label{mwithe}
For $k\in\mathbb{N}$, we say that $M$ is a \emph{manifold with $k$ ends} 
if there exists a compact set $K\subset M$ such that $M\setminus K$ 
is a disjoint union of $k$ connected components $E_{1},\dots,E_{k}$ and 
the closure of each $E_{i}$ is diffeomorphic to $S_{i}\times [R_{i},\infty)$, 
where $R_{i}>0$ and $S_{i}$ should be an $(n-1)$-dimensional compact manifold without boundary. 
Each $E_{i}$ is called the \emph{end} on $M$ and $S_{i}$ is called the \emph{link} of $E_{i}$. 
\end{definition}

Next, we give the definition of asymptotically Euclidean ends and asymptotically conical ends. 
Even though its definition can be written for general Riemannian manifolds, however, 
we give a definition under some assumptions 
which sufficiently work in our paper. 

\begin{definition}\label{AEAC}
Let $M$ be an $n$-dimensional manifold with $k$ ends $E_{1},\dots, E_{k}$. 
We assume that the link $S_{i}$ of $E_{i}$ is diffeomorphic to the $(n-1)$-dimensional standard sphere $S^{n-1}$ for each $i=1,\dots,k$. 
Let $g$ be a Riemannian metric on $M$. 
We assume that $g|_{E_{i}}$ is written as $g|_{E_{i}}=F_{i}(\rho_{i})g_{S^{n-1}}+(d\rho_{i})^2$ via 
the diffeomorphism $\overline{E_{i}}\cong S^{n-1}\times [R_{i},\infty)$ 
for some smooth positive function $F_{i}:[R_{i},\infty)\to\mathbb{R}$, where $\rho_{i}$ is the standard coordinate on $[R_{i},\infty)$. 
Then, we say that the end $E_{i}$ with metric $g$ is asymptotically Euclidean of order $\tau_{i}>0$ if 
$F_{i}$ satisfies 
\[
	(\partial_{\rho_{i}})^{j}( F_{i}(\rho_{i})-\rho_{i}^2)
	=\mathcal{O}(\rho_{i}^{-\tau_{i}-j})\qquad\mbox{ as } \rho_{i}\to \infty
\]
for all $j\in \mathbb{N}$. 
Similarly, we say that the end $E_{i}$ with metric $g$ is asymptotically conical of order $\tau_{i}>0$ if 
there exists $0<B_{i}<1$ such that 
$F_{i}$ satisfies 
\[
	(\partial_{\rho_{i}})^{j}( F_{i}(\rho_{i})-B_{i}\rho_{i}^2)
	=\mathcal{O}(\rho_{i}^{-\tau_{i}-j})\qquad\mbox{ as }\rho_{i}\to \infty
\]
for all $j\in\mathbb{N}$. In this case, $(S^{n-1},B_{i}g_{S^{n-1}})$ is called the link of $E_{i}$. 
\end{definition}

\subsection{Proof of (4)-(a)}
We check that $(\mathbb{R}^{n}\setminus\{0\},g_{t})$ 
is complete for all $t\in[0,\infty)$. 
Denote by $r:=|x|\in (0,\infty)$ the radius function from the origin. 
Since $u$ is radially symmetric for each $t\in[0,\infty)$ 
by (i) of Theorem \ref{th:blowdown}, 
we can write $u(x,t)$ as $u(r,t)$. 
In what follows, we abbreviate the variable $t$ when no confusion can arise. 
For instance, $u(r,t)$ is written as $u(r)$. 
We define another ``radius function'' $\rho=\rho(r):(0,\infty)\to\mathbb{R}$ by 
\[
\rho(r):=\int_{1}^{r}u^{\frac{2}{n+2}}(\tilde{r})d\tilde{r}.
\]
Then, $\rho$ is monotone increasing and changes sign at $r=1$. 
Indeed, this is the length (with sign) of a ray emerging from a point in the standard sphere in $\mathbb{R}^{n}\setminus\{0\}$ 
which diverges as $r\to \infty$ and converges to the origin as $r\to 0$. 
Then, by the well-known test for completeness, 
it suffices to prove that 
\begin{equation}\label{rhorho}
	\lim_{r\to 0}\rho(r)=-\infty, \qquad  
	\lim_{r\to\infty}\rho(r)=\infty
\end{equation}
to say that $(\mathbb{R}^{n}\setminus\{0\},g)$ is complete. 
However, this immediately follows from the asymptotic order of $u$ identified as (iii) and (v) of Theorem \ref{th:blowdown} 
with the fact that $\lambda>(n+2)/2\geq1$. 

Let $K:=\{\,x\in \mathbb{R}^n\setminus\{0\}; -1\leq \rho(|x|)\leq 1 \,\}$, 
and put $E_{1}:=\{\,x\in \mathbb{R}^n\setminus\{0\};\rho(|x|)>1 \,\}$ and 
$E_{2}:=\{\,x\in \mathbb{R}^n\setminus\{0\}; \rho(|x|)<-1 \,\}$. 
Then, it is clear that $K$ is compact and $\overline{E_{1}}$ and $\overline{E_{2}}$ are diffeomorphic to half cylinders. 
By the definition of $g=g_{t}$, it is trivial that $g$ is conformal to the standard metric on $\mathbb{R}^{n}\setminus\{0\}$. 
Then, the proof of (4)-(a) is complete.

\subsection{Proof of (1) and (2)}
We prove (1) and (2) for the initial metric in Theorem \ref{mainthm}. 
Since \eqref{rhorho} has been confirmed, for any $R\in\mathbb{R}$, we see that $\tilde{\rho}(r):=\rho(r)+R$ is a smooth diffeomorphism 
form $(0,\infty)$ to $(-\infty,\infty)$. 
Thus, we can write $r$ as a function of $\tilde{\rho}$ as $r=r(\tilde{\rho})$, 
and define a function $F=F(\tilde{\rho}):(-\infty,\infty)\to (0,\infty)$ by 
\begin{equation}\label{defFrho}
	F(\tilde{\rho})
	:=r^2 u^{\frac{4}{n+2}}(r) \qquad\mbox{ with } r=r(\tilde{\rho}). 
\end{equation}

Since the standard metric $g_{\mathbb{R}^{n}}$ is written as $g_{\mathbb{R}^{n}}=r^2g_{S^{n-1}}+(dr)^2$ 
by the standard metric $g_{S^{n-1}}$ of the $(n-1)$-dimensional standard sphere, 
the conformal metric $g=u^{4/(n+2)}g_{\mathbb{R}^{n}}$ 
can be expressed as 
$g=r^2u^{4/(n+2)}(r)g_{S^{n-1}}+u^{4/(n+2)}(r)(dr)^2$. 
By changing the radius coordinate from $r$ to $\tilde{\rho}$ and applying the relation $d\tilde{\rho}=\rho'(r)dr=u^{2/(n+2)}dr$, 
we see that 
\[
	g=F(\tilde{\rho})g_{S^{n-1}}+(d\tilde{\rho})^2. 
\]
This is nothing but the warped product of $(S^{n-1}, g_{S^{n-1}})$ and $\mathbb{R}$ with warping function $F(\tilde{\rho})$. 
Hence, it is diffeomorphic to the cylinder $S^{n-1}\times\mathbb{R}$ and the ``shape'' as a Riemannian manifold is determined 
by the asymptotic behavior of $F(\tilde{\rho})$ as $\tilde{\rho}\to -\infty$ and $\tilde{\rho} \to\infty$.

First, we prove (1), that is, $E_{1}$ is asymptotically Euclidean 
of order $m_{c}\lambda-2$ 
when $t=0$ under the assumption $n\geq 6$. 
Put 
\[
	R_{1}(r):=-\frac{2}{n+2}\int_{r}^{\infty}
	\tilde{r}u_{0}^{-\frac{n}{n+2}}(\tilde{r})
	\partial_{r}u_{0}(\tilde{r})d\tilde{r}, \qquad 
	R_{1}:=u_{0}^{\frac{2}{n+2}}(1)-R_{1}(1). 
\]
We note that $R_{1}(r)$ is well-defined, since $m_{c}\lambda>1$ and 
\begin{equation}\label{yesti1}
	\tilde{r}u_{0}^{-\frac{n}{n+2}}(\tilde{r})\partial_{r}u_{0}(\tilde{r})
	= 
	\mathcal{O}(\tilde{r}^{-m_{c}\lambda}) \qquad \mbox{ as } 
	\tilde{r}\to\infty.  
\end{equation}
Put 
\[\rho_{1}(r):=\rho(r)+R_{1}. \]
Then, a map sending $x\in \overline{E_{1}}$ to $(x/|x|,\rho_{1}(|x|))\in S^{n-1}\times[1+R_{1},\infty)$ gives a diffeomorphism. 
Setting $F(\rho_{1})$ by replacing $\tilde{\rho}$ in \eqref{defFrho} with $\rho_{1}$, we have 
$g_{0}|_{E_{1}}=F(\rho_{1})g_{S^{n-1}}+(d\rho_{1})^2$. 
We check the asymptotic behavior of $F$ ($=F_{0}$) on $E_{1}$. 
By integration by parts, we have 
\[
\begin{aligned}
\rho(r)&=
\left. \tilde{r}u_{0}^{\frac{2}{n+2}}(\tilde{r})\right|_{\tilde{r}=1}^{\tilde{r}=r}
-\frac{2}{n+2}\int_{1}^{r}\tilde{r}u_{0}^{-\frac{n}{n+2}}(\tilde{r})\partial_{r}u_{0}(\tilde{r})d\tilde{r}
	\\
&=ru_{0}^{\frac{2}{n+2}}(r)-u_{0}^{\frac{2}{n+2}}(1)-\frac{2}{n+2}\left(\int_{1}^{\infty}-\int_{r}^{\infty}\right)\tilde{r}u_{0}^{-\frac{n}{n+2}}(\tilde{r})\partial_{r}u_{0}(\tilde{r})d\tilde{r}\\
&=\sqrt{F(\rho_{1}(r))}-R_{1}(r)-R_{1}. 
\end{aligned}
\]
Then, by \eqref{yesti1} and $\rho_{1}(r)=\rho(r)+R_{1}$, we have 
\[
	\sqrt{F(\rho_{1}(r))}-\rho_{1}(r)
	=R_{1}(r)
	=\mathcal{O}(r^{-m_{c}\lambda+1})\qquad\mbox{ as } r\to\infty. 
\]
Taking the first derivative of the both hand sides with respect to $\rho_{1}$, we have 
\[\partial_{\rho_{1}}(\sqrt{F(\rho_{1})}-\rho_{1})=\partial_{r}R_{1} \cdot (\partial_{r}\rho_{1})^{-1}=\frac{2}{n+2}ru_{0}^{-1}\partial_{r}u_{0}=-\frac{2\lambda c_{1}^{m_{c}}}{n+2}\frac{r^{-m_{c}\lambda}}{c_{1}^{m_{c}}r^{-m_{c}\lambda}+c_{2}^{m_{c}}}. \]
Taking the $\rho_{1}$-derivative of this equality inductively, one can see that 
\[
	\partial_{\rho_{1}}^{\ell}(\sqrt{F(\rho_{1}(r))}-\rho_{1}(r))
	=\mathcal{O}(r^{-m_{c}\lambda+1-\ell})\qquad
	\mbox{ as }r\to\infty
\]
for all $\ell\geq 0$.
Since $r=\mathcal{O}(\rho_{1}(r))$ as $r\to\infty$, we have  
\[
	\partial_{\rho_{1}}^{\ell}(\sqrt{F(\rho_{1})}-\rho_{1})
	=\mathcal{O}(\rho_{1}^{-m_{c}\lambda+1-\ell})
	\qquad\mbox{ as }\rho_{1}\to\infty
\]
for all $\ell\geq 0$. 
This immediately implies that 
\[
	\partial_{\rho_{1}}^{\ell}(F(\rho_{1})-\rho_{1}^2)
	=\mathcal{O}(\rho_{1}^{-m_{c}\lambda+2-\ell})\qquad
	\mbox{ as }\rho_{1}\to\infty 
\]
for all $\ell\geq 0$. 
This proves (1).  
Remark that $m_{c}\lambda>(n-2)/2\geq 2 $ 
under the assumption $n\geq 6$. 

Next, we prove (2), that is, $E_{2}$ is asymptotically conical when $t=0$ for $n\geq 3$. 
Put 
\[
\begin{aligned}
	&R_{2}(r):=\frac{2}{n+2}
	\int_{0}^{r} u_{0}^{-\frac{n}{n+2}}(\tilde{r})
	(\lambda u_{0}(\tilde{r})+\tilde{r} \partial_{r}u_{0}(\tilde{r}))
	d\tilde{r}, \\
	&R_{2}
	:=\frac{1}{1-\frac{2\lambda}{n+2}}
	\left(-u_{0}^{\frac{2}{n+2}}(1)+R_{2}(1)\right). 
\end{aligned}
\]
We should check that $R_{2}(r)$ is well-defined. 
Clearly, we have 
\[
	u_{0}^{-\frac{n}{n+2}}(\tilde{r})
	=\mathcal{O}(\tilde{r}^{\frac{n}{n+2}\lambda}) 
	\qquad \mbox{ as } \tilde{r}\to0. 
\]
Moreover, by careful calculations, one can see that 
\[
	\lambda u_{0}(\tilde{r})+\tilde{r}\partial_{r}u_{0}(\tilde{r})
	=\lambda\frac{c_2^{m_{c}}\tilde{r}^{m_{c}\lambda}}{c_{1}^{m_{c}}+c_{2}^{m_{c}}\tilde{r}^{m_{c}\lambda}}
	(c_{1}^{m_{c}}\tilde{r}^{-m_{c}\lambda}+c_{2}^{m_{c}})^{\frac{1}{m_{c}}}
	=\mathcal{O}(\tilde{r}^{m_{c}\lambda-\lambda})
	\qquad\mbox{ as }\tilde{r}\to 0. 
\]
These imply that 
\begin{equation}\label{yesti2}
u_{0}^{-\frac{n}{n+2}}(\tilde{r})(\lambda u_{0}(\tilde{r})+\tilde{r}\partial_{r}u_{0}(\tilde{r}))=\mathcal{O}(\tilde{r}^{\frac{n-4}{n+2}\lambda}) 
\qquad \mbox{ as } \tilde{r}\to 0. 
\end{equation}
Thus, the integral of the left hand side of \eqref{yesti2} on $(0,r]$ is finite 
since $(n-4)\lambda/(n+2)> -1$. 
Put 
\[\rho_{2}(r):=-\rho(r)+R_{2}. \]
Then, a map sending $x\in \overline{E_{2}}$ to $(x/|x|,\rho_{2}(|x|))\in S^{n-1}\times[1+R_{2},\infty)$ gives 
a diffeomorphism. 
Setting 
\[
	F(\rho_{2}):=r^2 u^{\frac{4}{n+2}}(r) 
	\qquad\mbox{ with } \rho_{2}(r)=\rho_{2},  
\]
we have 
$g_{0}|_{E_{2}}=F(\rho_{2})g_{S^{n-1}}+(d\rho_{2})^2$. 
We check the asymptotic behavior of $F$ ($=F_{0}$) on $E_{2}$. 
By integration by parts, we have 
\[
\begin{aligned}
	\rho(r)
	&=\int_{1}^{r}(\tilde{r}^{\lambda}u_{0}(\tilde{r}))^{\frac{2}{n+2}}
	\tilde{r}^{-\frac{2\lambda}{n+2}}d\tilde{r}\\
	&=\frac{1}{1-\frac{2\lambda}{n+2}}
	\left(\left. 
	\tilde{r}u_{0}^{\frac{2}{n+2}}(\tilde{r})
	\right|_{\tilde{r}=1}^{\tilde{r}=r}
	-\frac{2}{n+2}\left(\int_{1}^{0}-\int_{r}^{0}\right)
	u_{0}^{-\frac{n}{n+2}}(\tilde{r})
	(\lambda u_{0}(\tilde{r})+\tilde{r} \partial_{r}u_{0}(\tilde{r}))
	d\tilde{r} \right)\\
	&=\frac{1}{1-\frac{2\lambda}{n+2}}\left(
	\sqrt{F(\rho_{2}(r))}
	-R_{2}(r)\right)+R_{2}. 
\end{aligned}
\]
Then, by \eqref{yesti2} and $\rho_{2}(r)=-\rho(r)+R_{2}$, we see that 
\begin{equation}\label{asympE2}
\sqrt{F(\rho_{2}(r))}-\left(\frac{2\lambda}{n+2}-1\right)\rho_{2}(r)
	=R_{2}(r)
	=\mathcal{O}(r^{\frac{n-4}{n+2}\lambda+1})
	\qquad\mbox{ as }r\to 0. 
\end{equation}
Since 
$r^{1-{2\lambda/(n+2)}}=\mathcal{O}(\rho_{2}(r))$ as $r\to 0$, we have 
\[
	\sqrt{F(\rho_{2})}-\sqrt{B}\rho_{2}
	=\mathcal{O}(\rho_{2}^{-(\tau+1)})\qquad\mbox{ as }
	 \rho_{2}\to\infty,
\]
where 
\begin{equation}\label{Btau}
	B:=\left(\frac{2}{n+2}\lambda-1\right)^2\in(0,1), \qquad 
	\tau:=
	\frac{\frac{n-4}{n+2}\lambda+1}{\frac{2}{n+2}\lambda-1}-1
	=\frac{\frac{n-6}{n+2}\lambda+2}{\frac{2}{n+2}\lambda-1}>0. 
\end{equation}
Taking the first derivative of \eqref{asympE2} with respect to $\rho_{2}$, one can see that 
\[\partial_{\rho_{2}}(\sqrt{F(\rho_{2})}-\sqrt{B}\rho_{2})=\partial_{r}R_{2} \cdot (\partial_{r}\rho_{2})^{-1}=-\frac{2\lambda c_{2}^{m_{c}}}{n+2}(c_{1}^{m_{c}}r^{-m_{c}\lambda}+c_{2}^{m_{c}})^{-1}. \]
Taking the $\rho_{2}$-derivative of this equality inductively, one can see that 
\[
	\partial_{\rho_{2}}^{\ell}(\sqrt{F(\rho_{2}(r))}-\sqrt
	{B}\rho_{2}(r))=\mathcal{O}(r^{\alpha_{\ell}})
	\qquad\mbox{ as }r\to 0
\]
for all $\ell\geq 0$, where 
\[\alpha_{\ell}:=\left(\frac{2}{n-2}(\ell-1)+\ell \right)m_{c}\lambda+(\ell-1)(-m_{c}\lambda-1).\]
By straightforward computations, we can see that 
\[\alpha_{\ell}=\left(\frac{2}{n+2}\lambda-1\right)\ell+\frac{n-4}{n+2}\lambda+1.\]
This equality with $r^{1-{2\lambda/(n+2)}}=\mathcal{O}(\rho_{2}(r))$ implies that 
\[
	\partial_{\rho_{2}}^{\ell}(\sqrt{F(\rho_{2})}-\sqrt
	{B}\rho_{2})
	=\mathcal{O}(\rho_{2}^{-(\tau+1)-\ell})
	\qquad\mbox{ as }\rho_{2}\to\infty
\]
for all $\ell\geq 0$. 
This immediately implies that 
\[
	\partial_{\rho_{2}}^{\ell}(F(\rho_{2})-B\rho_{2}^2)
	=\mathcal{O}(\rho_{2}^{-\tau-\ell})
	\qquad\mbox{ as }\rho_{2}\to\infty. 
\]
This proves (2), that is, $E_{2}$ with the initial metric $g_{0}$ is asymptotically conical 
of order $\tau>0$ in the sense of $C^{\infty}$. 
By looking at $B$, we can roughly imagine that the end $E_{2}$ tends to be an asymptotically Euclidean end when the parameter 
$\lambda\in ((n+2)/2,n+2)$ approaches to $n+2$ and, 
contrary, it tends to be a narrower horn when $\lambda$ approaches to $(n+2)/2$. 

\subsection{Proof of (3)}
To prove (3), we calculate 
the scalar curvature of $g_{0}=u_{0}^{4/(n+2)}g_{\mathbb{R}^n}$. 
By the general formula \eqref{scalc} with $v:=c_{1}^{m_{c}}|x|^{-m_{c}\lambda}+c_{2}^{m_{c}}$ and $\scal(g_{\mathbb{R}^{n}})=0$, 
we have
\[
\begin{aligned}
	\mathrm{scal}(g_{0})&=-\frac{4(n-1)}{n-2}v^{-\frac{n+2}{n-2}}(r)
	\left(\partial_{r}^2v(r)+\frac{1}{r}\partial_{r}v(r)\right)\\
	&=-\frac{4(n-1)}{n-2}v^{-\frac{n+2}{n-2}}(r)c_{1}^{m_{c}}
	(m_{c}\lambda)^2r^{-m_{c}\lambda-2}<0, 
\end{aligned}
\]
and (3) is proved.

\subsection{Proof of (4)}
The property (4) is about the Yamabe flow starting from $g_{0}$. 
We remark that the existence and uniqueness of a long time solution is already proved as we wrote at the beginning of this section. 
We also remark that the property (4)-(a) is already proved. 
Thus, we prove (4)-(b), (4)-(c) and (4)-(d). 
However, (4)-(d) clearly follows from (vi) of Theorem \ref{th:blowdown}. 
Hence, it suffices to prove (4)-(b) and (4)-(c). 

First, we prove (4)-(b). 
Since the volume form of $g_{t}=u^{4/(n+2)}g_{\mathbb{R}^n}$ is given by 
\[\mathop{\mathrm{vol}}(g_{t})=u^{\frac{2n}{n+2}}\mathop{\mathrm{vol}}(g_{\mathbb{R}^n}), \]
we can easily see that (4)-(b) holds by (iii) and (v) of Theorem \ref{th:blowdown}. 

Next, we prove (4)-(c), that is, the Yamabe constant $Y(M,g_{t})$ of $(M,g_{t})$ 
is equal to $Y(S^{n},g_{S^{n}}) (>0)$ for all $t\in[0,\infty)$. 
We remark that for a general Riemannian manifold $(M,g)$ 
(not necessarily compact), its Yamabe constant is defined as follows. 

\begin{definition}\label{Yaminv}
Let $(M,g)$ be an $n$-dimensional Riemannian manifold. 
Then, the Yamabe constant of $(M,g)$ is defined by 
\[Y(M,g):=\inf_{f\in C_0^{\infty}(M)}\frac{4\frac{n-1}{n-2}\int_{M}|\nabla f|^{2}dv_{g}+\int_{M}\scal(g)f^2dv_{g}}{\left(\int_{M}|f|^{\frac{2n}{n-2}}dv_{g}\right)^{\frac{n-2}{n}}}, \]
where $C_{0}^{\infty}(M)$ is the set of all smooth functions 
with compact support on $M$. 
\end{definition}

Then, we have $Y(M,g_{t})=Y(\mathbb{R}^{n}\setminus\{0\},g_{\mathbb{R}^{n}})$ since our initial Riemannian manifold $(M,g_{t})$ 
is conformal to $(\mathbb{R}^{n}\setminus\{0\},g_{\mathbb{R}^{n}})$ 
and the Yamabe constant is conformally invariant. 
Let $N$ and $S$ be the North and South poles of $S^{n}$. 
Since the stereographic projection gives a conformal isometry between $(\mathbb{R}^{n},g_{\mathbb{R}^{n}})$ and $(S^{n}\setminus\{N\},g_{S^{n}})$, 
$(\mathbb{R}^{n}\setminus\{0\},g_{\mathbb{R}^{n}})$ 
is conformal to $(S^{n}\setminus\{S,N\},g_{S^{n}})$. 
Thus, $Y(\mathbb{R}^{n}\setminus\{0\},g_{\mathbb{R}^{n}})=Y(S^{n}\setminus\{S,N\},g_{S^{n}})$. 
Since it is known that $Y(U,g_{S^{n}})=Y(S^{n},g_{S^{n}})$ for an open subset $U$ of $S^{n}$ (see \cite[Lemma 2.1]{SY88}), 
we have $Y(S^{n}\setminus\{S,N\},g_{S^{n}})=Y(S^{n},g_{S^{n}})=n(n-1)\Vol(S^{n},g_{S^{n}})^{2/n}$, where the last equality follows from \cite{A76}. 
Thus, we have 
\[
	Y(M,g_{t})=Y(S^{n},g_{S^{n}})=n(n-1)\Vol(S^{n},g_{S^{n}})^{\frac{2}{n}}>0. 
\]
Hence, (4) is proved. 
The proof of Theorem \ref{mainthm} is complete.

\end{document}